\documentclass[10 pt]{amsart}
\usepackage[scale={0.70,0.80},centering,includeheadfoot]{geometry}
\usepackage{amsmath,amssymb}

 \newtheorem{thm}{Theorem}[section]
 \newtheorem{cor}[thm]{Corollary}
 \newtheorem{lem}[thm]{Lemma}
 
 \theoremstyle{definition}
 \newtheorem{defn}[thm]{Definition}
 \newtheorem{exam}{Example}[section]
 \newtheorem{rem}{Remark}[section]
 \newtheorem{quest}{Question}[section]
 \numberwithin{equation}{section}

\DeclareMathOperator{\tr}{tr}
\DeclareMathOperator{\di}{div}

\DeclareMathOperator{\grad}{grad}
\newcommand{\Rom}{R}

\newcommand{\V}{\mathcal{V}}
\DeclareMathOperator{\Ker}{Ker}
\newcommand{\Hi}{\mathcal{H}}

\newcommand{\Si}{\mathbb{S}}
\newcommand{\Real}{\mathbb{R}}
\newcommand{\Vol}{Vol}

\newcommand{\X}{\mathfrak{X}}

\begin{document}

\title{The biharmonic homotopy problem for unit vector fields on 2-tori}

\author{E.~Loubeau and M.~Markellos}

\address{D\'{e}partement de Math\'{e}matiques, LMBA, UMR 6205, Universit\'{e} de Bretagne Occidentale, 6, Avenue Victor Le Gorgeu, CS 93837, 29238 Brest Cedex 3, France}
\email{Eric.Loubeau@univ-brest.fr}

\address{Department of Mathematics and Statistics, University of Cyprus, P.O. Box 20537, 1678 Nicosia, Cyprus}
\email{mmarkellos@hotmail.gr}

%\dedicatory{}
%%% ----------------------------------------------------------------------
\maketitle

\begin{abstract}
The bienergy of smooth maps between Riemannian manifolds, when restricted to unit vector fields, yields two different variational problems depending on whether one takes the full functional or just the vertical contribution. Their critical points, called biharmonic unit vector fields and biharmonic unit sections, form different sets. Working with surfaces, we first obtain general characterizations of biharmonic unit vector fields and biharmonic unit sections under conformal change of the metric. In the case of a 2-dimensional torus, this leads to a proof that biharmonic unit sections are always harmonic and a general existence theorem, in each homotopy class, for biharmonic unit vector fields.
\end{abstract}

\section{Introduction}

One of the most studied object in Differential Geometry is the energy functional of a map $\varphi:(M,g)\rightarrow (N, h)$ between
Riemannian manifolds, given by
\begin{equation*}
E_{1}(\varphi)=\frac{1}{2}\int_{M}{\|d\varphi\|^{2}v_{g}}
\end{equation*}
where $d\varphi$ denotes the differential of the map $\varphi$ (\cite{Baird03, EellsSampson64}). Its critical
points are called \textit{harmonic maps}
and are characterized by the vanishing of the \textit{tension
field} $\tau_{1}(\varphi)=\tr \nabla d\varphi$.

Given a Riemannian manifold $(P, k)$ and a Riemannian submersion $\pi:(P, k)\rightarrow (N, h)$ with $\V:TP\rightarrow \Ker d\pi$ the projection on the vertical subbundle, the \emph{vertical energy} of a section $\sigma:M \rightarrow P$ has been defined by C.M.Wood \cite{Wood86} as
\begin{equation*}
E_{1}^{v}(\sigma)=\frac{1}{2}\int_{M}\|\V(d\sigma)\|^{2}v_{g}.
\end{equation*}

Denote by $(TM, g_{S})$ the tangent bundle of $(M, g)$ equipped with the Sasaki metric $g_{S}$. A vector field $X$ on $M$ determines a map from $(M, g)$ into $(TM, g_{S})$ (i.e. a section of $TM$), embedding $M$ into its tangent bundle. It was shown in \cite{Ishiha79} and \cite{Nouhaud77} that if $M$ is compact and $X$ is a harmonic map from $(M, g)$ into $(TM, g_{S})$, then $X$ must be parallel. This rigidity may be overcome for vector fields of unit length. A unit vector field $V$ on $(M, g)$ is a section of the unit tangent sphere bundle $T_{1}M,$ i.e. $V:(M, g)\rightarrow (T_{1}M, g_{S})$, where $T_{1}M$ carries the metric induced from $g_{S}$. Its \emph{energy} $E_{1}(V)$ (resp. \emph{vertical energy} $E_{1}^{v}(V)$) is given by
\begin{equation*}
\begin{array}{lr}
E_{1}(V)=\frac{m}{2}\Vol(M)+\frac{1}{2}\int_{M}\|\nabla V\|^{2}v_{g}, & E_{1}^{v}(V)=\frac{1}{2}\int_{M}\|\nabla V\|^{2}v_{g},
\end{array}
\end{equation*}
where $m=\dim M$ \cite{Wiegm95}. Critical points of $E_{1}$, respectively $E_{1}^{v}$, with respect to variations through unit vector fields are called \emph{harmonic unit vector fields}, respectively \emph{harmonic unit sections}. While both functionals clearly give the same critical points, this will no longer be the case for our problem. The corresponding critical point condition has been determined in \cite{Wiegm95} and \cite{Wood97}. It should be pointed out that a harmonic unit vector field determines a harmonic map when an additional condition involving the curvature is satisfied (cf. \cite{GilMedr01, HanYim98}). Wiegmink (\cite{Wiegm95}) studied harmonic unit vector fields on a 2-dimensional torus. Harmonic unit vector fields have acquired an extensive literature, see for example the bibliography of \cite{DragomirPerro11}.

In \cite[p.77]{EellsLemai83}, Eells and Lemaire suggested studying polyharmonic maps of order $k$ between Riemannian manifolds. Of particular interest is the case
$k=2$. They defined the \textit{bienergy} of a map $\varphi$ between Riemannian manifolds as the functional
\begin{equation*}
E_{2}(\varphi)=\frac{1}{2}\int_{M}{\|\tau_{1}(\varphi)\|^{2}v_{g}},
\end{equation*}
and a smooth map $\varphi$ is \textit{biharmonic} if it is a critical
point of the functional $E_{2}$. In \cite{Jiang08}, Jiang first derived the associated
Euler-Lagrange equation for $E_{2}$. A harmonic map is automatically a biharmonic map, and non-harmonic biharmonic maps are called
\textit{proper biharmonic}. Since 2001, biharmonic submanifolds of Riemannian space forms have been vigorously studied by many geometers in order to classify and/or characterize them (\cite{CaMoOn01}). We refer to the survey (\cite{Ou16}) for some history, fundamental problems, current results and open problems on biharmonic submanifolds of Riemannian space forms.

In \cite{MarkUra}, Markellos and Urakawa defined the bienergy of a vector field $X$ on a Riemannian manifold $(M, g)$ as the bienergy of the corresponding map $X:(M, g)\mapsto (TM, g_{S})$. More precisely, they proved that if $(M, g)$ is compact and $X$ is biharmonic map, then $X$ must be parallel. Furthermore, this remains the case under the less stringent and arguably more natural condition that $X$ is a \emph{biharmonic vector field} i.e. a critical point of the bienergy functional with respect to variations through vector fields. The situation is very different when we consider the set of unit vector fields on $(M, g)$. The \emph{bienergy} (resp. \emph{vertical bienergy}) of a unit vector field $V$ is defined by
\begin{eqnarray}
E_{2}(V)&=&\frac{1}{2}\int_{M}\|\tau_{1}(V)\|^{2}v_{g}=\frac{1}{2}\int_{M}\Large \left
\{g(S(V), S(V))+g(\bar{\Delta}V,
\bar{\Delta}V)-(g(\bar{\Delta}V, V))^{2}\right\}v_{g},\nonumber\\
E_{2}^{v}(V)&=&\frac{1}{2}\int_{M}\|\tau_{1}^{v}(V)\|^{2}v_{g}=\frac{1}{2}\int_{M}\Big\{g(\bar{\Delta}V,
\bar{\Delta}V)-(g(\bar{\Delta}V, V))^{2}\Big\}v_{g},\nonumber
\end{eqnarray}
where $S(V)=\sum_{i=1}^{m}R(\nabla_{e_{i}}V, V)e_{i}$ and $\bar{\Delta}V=\sum_{i=1}^{m}\{\nabla_{\nabla_{e_{i}}{e_{i}}}V-\nabla_{e_{i}}\nabla_{e_{i}}V\}$ ($\{e_{i}\}_{i=1}^{m}$ a local orthonormal frame field of $(M, g)$). Critical points of $E_{2}$ (resp. $E_{2}^{v}$) with respect to variations through unit sections are called \emph{biharmonic unit vector fields } (resp. \emph{biharmonic unit sections}). In \cite{MarkUra14}, the authors derived the Euler-Lagrange equation for biharmonic unit vector fields and completely determined the class of invariant biharmonic unit vector fields on three-dimensional unimodular Lie groups and the generalized Heisenberg groups. We should mention that the majority of these examples are on non-compact Riemannian manifolds. The main goal of this paper is to study biharmonic unit vector fields and unit sections on surfaces, particularly tori. In this paper, we address the following questions:
\begin{quest}
Do there exist examples of biharmonic unit sections?
\end{quest}
\begin{quest}
Is there a relation between the notions of ''harmonic unit section'', a ''biharmonic unit section'' and ''biharmonic unit vector field''?
\end{quest}
\begin{quest}
Do biharmonic unit sections and biharmonic unit vector fields exist on a 2-dimensional torus?
\end{quest}
We should point out that a harmonic unit section is automatically a biharmonic unit section. In contrast with the harmonic case, we show that the notions of ''biharmonic unit vector field'' and ''biharmonic unit section'' are not equivalent (see Remark 3.1). As to Question 1.3, we show that biharmonic unit sections and unit vector fields are not invariant under conformal change of the metric on a 2-torus. This is a significant difference with the set of harmonic unit vector fields on $(T^{2}, g)$ which is conformally invariant (\cite{Wiegm95}).

The paper is organized in the following way: Section 2 contains the presentation of some basic notions such as the geometry of the unit tangent sphere, (bi)harmonic maps between Riemannian manifolds and the notions of ''biharmonic unit section'' and ''biharmonic unit vector field''. In Section 3, we give examples of non-harmonic biharmonic unit sections on Riemannian manifolds (see Theorems 3.1, 3.2 and 3.3). In Section 4, we derive the critical points of the bienergy functional (resp. vertical bienergy functional) for unit vector fields after a conformal change of the metric on the base manifold (see Theorems 4.3 and 4.4). In Section 5, we study biharmonic unit sections and unit vector fields on a 2-torus $T^{2}$ (see Theorems 5.3 and 5.6). Furthermore, we investigate the stability of biharmonic unit vector fields on $(T^{2}, g)$ (see Theorem 5.7).

\section{Preliminaries}

\subsection{Harmonic - Biharmonic maps}
Let $(M, g), (N, h)$ be Riemannian manifolds of dimensions $m$ and $n$, respectively, and let $\varphi:
(M, g)\rightarrow (N, h)$ be a smooth map between them. The
\textit{energy density of} $\varphi$ is the smooth function
$e(\varphi):M \rightarrow [0,\infty)$ given by
\begin{equation*}
e(\varphi)(x)=\frac{1}{2}\|d\varphi_{x}\|^{2}=\frac{1}{2}\sum_{i=1}^{m}{h(d\varphi_{x}(e_{i}), d\varphi_{x}(e_{i}))},
\end{equation*}
where $x \in M$, $\{e_{i}\}_{i=1}^{m}$ is an orthonormal basis for $T_{x}M$ and
$d\varphi_{x}$ denotes the differential of the map $\varphi$ at the
point $x$. Furthermore, we denote by $C^{\infty}(M, N)$ the space of
all smooth maps from $M$ to $N$, $\nabla^{\varphi}$ the connection
of the vector bundle $\varphi^{-1}TN$ induced from the Levi-Civita
connection $\bar{\nabla}$ of $(N, h)$ and $\nabla$ the Levi-Civita
connection of $(M, g)$.

Let $D$ be a compact domain of $M$. The \textit{energy (integral) of
$\varphi$} over $D$ is the integral of its energy density:
\begin{equation*}
E_{1}(\varphi; D)=\frac{1}{2}\int_{D}\|d\varphi\|^{2}v_{g}=\int_{D}e(\varphi)v_{g}.
\end{equation*}
If $M$ is compact, we write $E(\varphi)$ instead of $E(\varphi; D)$.

A map $\varphi: (M, g)\rightarrow (N, h)$ is said to be \textit{harmonic} if it is a critical point of the energy functional
$E(.;D):C^{\infty}(M, N)\rightarrow \Real$ for any compact domain $D$.
It is well known (\cite{EellsSampson64}) that the map $\varphi:
(M, g)\rightarrow (N, h)$ is harmonic if and only if
\begin{equation*}
\tau_{1}(\varphi)=\tr(\nabla
d\varphi)=\sum_{i=1}^{m}\{\nabla_{e_{i}}^{\varphi}d\varphi(e_{i})-d\varphi(\nabla_{e_{i}}{e_{i}})\}=0,
\end{equation*}
where $\{e_{i}\}_{i=1}^{m}$ is a local orthonormal frame field of $(M, g)$. The
equation $\tau_{1}(\varphi)=0$ is called the \emph{harmonic equation}.

Eells and Lemaire \cite[p.77]{EellsLemai83} introduced the
notion of polyharmonic maps. In this paper, we only consider
polyharmonic maps of order two, usually called
\emph{biharmonic maps}.

A smooth map $\varphi: (M, g)\rightarrow
(N, h)$ is said to be \emph{biharmonic} if it is a critical
point of the bienergy functional:
\begin{equation*}
E_{2}(\varphi)=\frac{1}{2}\int_{D}\|\tau_{1}(\varphi)\|^{2}v_{g},
\end{equation*}
over every compact region $D$ of $M$. The corresponding Euler-Lagrange equation associated to the bienergy functional
becomes more complicated and it involves the curvature of the
target manifold $(N, h)$ (\cite{Jiang08}). We should mention that a harmonic map is automatically a
biharmonic map, in fact a minimum of the bienergy functional. Many examples of
non-harmonic biharmonic maps have been obtained in
\cite{CaMoOn01}, \cite{CaMoOn02}, \cite{FetcuLouOnic17}, \cite{LoubeauOniciuc16} and
\cite{Ou16}. Non-harmonic biharmonic maps are called \emph{proper} biharmonic maps.

\subsection{The tangent bundle and the unit tangent sphere
bundle}

Let $(M, g)$ be an $m$-dimensional Riemannian manifold and $\nabla$
the associated Levi-Civita connection. Its Riemann curvature tensor
$R$ is defined by
\begin{equation*}
R(X, Y)Z =\nabla_{X}\nabla_{Y}Z - \nabla_{Y}\nabla_{X}Z -
\nabla_{[X, Y]}Z
\end{equation*}
for all vector fields $X, Y$ and $Z$ on $M$. The tangent bundle of $(M, g)$, denoted by $TM$, consists of
pairs $(x, u)$ where $x$ is a point in $M$ and $u$ a tangent vector
to $M$ at $x$. The mapping $\pi: TM \rightarrow M: (x, u)\mapsto x$ is
the natural projection from $TM$ onto $M$. The tangent space $T_{(x, u)}TM$ at a point $(x,
u)$ in $TM$ is a direct sum of the vertical subspace $\V_{(x,
u)}=Ker(d\pi|_{(x, u)})$ and the horizontal subspace $\Hi_{(x,
u)}$, with respect to the Levi-Civita connection $\nabla$ of $(M, g)$:
\begin{equation*}
T_{(x, u)}TM=\Hi_{(x, u)}\oplus \V_{(x, u)}.
\end{equation*}
For any vector $w\in T_{x}M$, there exists a unique vector $w^{h}\in
\Hi_{(x, u)}$ at the point $(x, u)\in TM$, which is called the
\emph{horizontal lift} of $w$ to $(x, u)$, such that $d\pi(w^{h})=w$
and a unique vector $w^{v}\in \V_{(x, u)}$, the
\emph{vertical lift} of $w$ to $(x, u)$, such that $w^{v}(df)=w(f)$
for all functions $f$ on $M$. Hence, every tangent vector
$\bar{w}\in T_{(x, u)}TM$ can be decomposed as
$\bar{w}=w_{1}^{h}+w_{2}^{v}$ for uniquely determined vectors
$w_{1}, w_{2} \in T_{x}M$. The \emph{horizontal} (respectively,
\emph{vertical}) \emph{lift} of a vector field $X$ on $M$ to $TM$ is
the vector field $X^{h}$ (respectively, $X^{v}$) on $TM$ whose value
at the point $(x, u)$ is the horizontal (respectively, vertical)
lift of $X_{x}$ to $(x, u)$.

The tangent bundle $TM$ of a Riemannian manifold $(M, g)$ can be
endowed in a natural way with a Riemannian metric $g_{S}$, the
\emph{Sasaki metric}, depending only on the Riemannian structure $g$
of the base manifold $M$. It is uniquely determined by
\begin{equation}\label{Eq: Sasaki metric}
\begin{array}{ll}
g_{S}(X^{h}, Y^{h})=g_{S}(X^{v}, Y^{v})=g(X, Y)\circ \pi, &g_{S}(X^{h}, Y^{v})=0\\
\end{array}
\end{equation}
for all vector fields $X$ and $Y$ on $M$. More intuitively, the
metric $g_{S}$ is constructed in such a way that the vertical and
horizontal subbundles are orthogonal and the bundle map $\pi: (TM,
g_{S})\mapsto (M, g)$ is a Riemannian submersion.

Next, we consider the unit tangent sphere bundle $T_{1}M$ which is
an embedded hypersurface of $TM$ defined by the equation $g_{x}(u,
u)=1$. The vector field $N_{(x, u)}=u^{v}$ is
a unit normal of $T_{1}M$ as well as the position vector for a point
$(x, u)\in T_{1}M$. For $X\in T_{x}M$, we define the
\emph{tangential lift} of $X$ to $(x, u)\in T_{1}M$ by
(\cite{Boe97})
\begin{equation*}
X^{t}_{(x, u)}=X^{v}_{(x, u)}-g(X, u)N_{(x, u)}.
\end{equation*}
For the sake of notational clarity, we will use $\bar{X}$ as a
shorthand of $X-g(X,u)u$. Then $X^{t}=\bar{X}^{v}$. Clearly, the
tangent space to $T_{1}M$ at $(x, u)$ is spanned by vectors of the
form $X^{h}$ and $X^{t}$ where $X\in T_{x}M$. The \emph{tangential
lift} of a vector field $X$ on $M$ to $T_{1}M$ is the vertical
vector field $X^{t}$ on $T_{1}M$ whose value at the point $(x, u)\in
T_{1}M$ is the tangential lift of $X_{x}$ to $(x, u)$. In the following, we consider
$T_{1}M$ with the metric induced from the Sasaki
metric $g_{S}$ of $TM$, which is also denoted by $g_{S}$.
We denote by $\X(M)$ (resp. $\X^{1}(M)$) the set of globally defined vector fields (resp. unit vector fields) on the base manifold $(M, g)$.

In the sequel, we concentrate on the map $V: (M, g)\rightarrow (T_{1}M,
g_{S})$. The tension field $\tau_{1}(V)$ is given by (\cite{GilMedr01})
\begin{equation}\label{Eq: tension field of V}
\tau_{1}(V)=(-S(V))^{h}+(-\bar{\Delta}V)^{t},
\end{equation}
where $\{e_{i}\}_{i=1}^{m}$ is a local orthonormal frame field of $(M, g)$, $S(V)=\sum_{i=1}^{m}{\Rom(\nabla_{e_{i}}{V}, V)e_{i}}$ and $\bar{\Delta}V=-\tr\nabla^{2}V=\sum_{i=1}^{m}\{\nabla_{\nabla_{e_{i}}{e_{i}}}V-\nabla_{e_{i}}\nabla_{e_{i}}V\}$.

Let now $(M, g)$ be a compact $m$-dimensional Riemannian manifold and $V \in \X^{1}(M).$ The \emph{bienergy} (resp. \emph{vertical bienergy}) $E_{2}(V)$ (resp. $E_{2}^{v}(V)$ of $V$ is defined as the bienergy (resp. vertical bienergy) associated with the corresponding map $V:(M, g)\rightarrow (T_{1}M, g_{S})$. More precisely, by using relations (\ref{Eq: Sasaki metric}) and
(\ref{Eq: tension field of V}), we get
\begin{eqnarray}
E_{2}(V)&=&\frac{1}{2}\int_{M}\|\tau_{1}(V)\|^{2}v_{g}=\frac{1}{2}\int_{M}\Large \left
\{g(S(V), S(V))+g(\bar{\Delta}V,
\bar{\Delta}V)-(g(\bar{\Delta}V, V))^{2}\right\}v_{g},\nonumber\\
E_{2}^{v}(V)&=&\frac{1}{2}\int_{M}\|\tau_{1}^{v}(V)\|^{2}v_{g}=\frac{1}{2}\int_{M}\Large \left \{g(\bar{\Delta}V,
\bar{\Delta}V)-(g(\bar{\Delta}V, V))^{2}\right\}v_{g}.\nonumber
\end{eqnarray}
Now, we give the following definitions:
\begin{defn}
Let $(M, g)$ be a Riemannian manifold. A unit vector field
$V\in \X^{1}(M)$ is a \emph{biharmonic unit vector field} (resp. \emph{biharmonic unit section}) if the corresponding
map $V:(M, g)\rightarrow (T_{1}M, g_{S})$ is a critical point for the
bienergy functional $E_{2}$ (resp. $E_{2}^{v}$), considering only variations through maps
defined by unit vector fields.
\end{defn}
The critical point condition for the bienergy functional $E_{2}$ restricted to $\X^{1}(M)$ has been determined in \cite{MarkUra14}. More precicely,
\begin{thm}
Let $(M, g)$ be an  $m$-dimensional Riemannian manifold. Then, $V\in \X^{1}(M)$ is a biharmonic unit vector field if and only if
\begin{eqnarray}\label{Eq: biharmonic unit vector field}
\sum_{i=1}^{m}\Large \left\{R(e_{i},
\nabla_{e_{i}}S(V))V+(\nabla_{e_{i}}R)(e_{i}, S(V))V
+2R(e_{i}, S(V))\nabla_{e_{i}}V\right\}+\nonumber\\
+\bar{\Delta}[\bar{\Delta}V-g(\bar{\Delta}V, V)V]-g(\bar{\Delta}V,
V)\bar{\Delta}V
\end{eqnarray}
is collinear to $V$, where $\{e_{i}\}_{i=1}^{m}$ is a local orthonormal frame field of $(M, g)$.
\end{thm}
Following the terminology used in \cite{MarkUra14}, we get
\begin{thm}
Let $(M, g)$ be an  $m$-dimensional Riemannian manifold. Then, $V\in \X^{1}(M)$ is a biharmonic unit section if and only if
\begin{eqnarray}\label{Eq: biharmonic unit section}
\bar{\Delta}[\bar{\Delta}V-g(\bar{\Delta}V, V)V]-g(\bar{\Delta}V,
V)\bar{\Delta}V
\end{eqnarray}
is collinear to $V$.
\end{thm}
If a unit vector field $V$ of a Riemannian manifold $(M, g)$ defines a harmonic map from $(M, g)$ into $(T_{1}M, g_{S})$ i.e. $S(V)=0$ and $\bar{\Delta}V \hspace{0.10cm}$ is collinear to $V$, then it is automatically a biharmonic unit vector field and biharmonic unit section. Furthermore, if $V$ is a harmonic unit vector field (equivalently, harmonic unit section) i.e. $\bar{\Delta}V$ is collinear to $V$, then it is automatically a biharmonic unit section but, not necessarily a biharmonic unit vector field, unless the manifold is flat.

\section{Examples of biharmonic unit sections which are not harmonic sections}

\begin{exam} We consider the compact and simple unimodular Lie group $SU(2)$ equipped with a left-invariant Riemannian metric $< , >$. Then, there there exists an orthonormal basis $\{e_{1}, e_{2}, e_{3}\}$ of the Lie algebra $su(2)$ such that (\cite{Milnor76})
\begin{equation}\label{Eq:Liebracketsunimodular}
\begin{array}{lll}
[e_{2}, e_{3}]=\lambda_{1}e_{1}, &[e_{3}, e_{1}]=\lambda_{2}e_{2}, &[e_{1}, e_{2}]=\lambda_{3}e_{3},
\end{array}
\end{equation}
where $\lambda_{1}, \lambda_{2}, \lambda_{3}$ are strictly positive constants with $\lambda_{1}\geq \lambda_{2}\geq \lambda_{3}$.
\end{exam}
Let $\theta^{i}, i=1, 2, 3$, be the dual one forms of $\{e_{i}\},
i=1, 2, 3$. Then the Levi-Civita connection $\nabla$ is determined
by
\begin{eqnarray}\label{Eq:nabla SU(2)}
\nabla e_{1}&=&\mu_{3}e_{2}\otimes\theta^{3}-\mu_{2}e_{3}\otimes \theta^{2},\nonumber\\
\nabla e_{2}&=&-\mu_{3}e_{1}\otimes \theta^{3}+\mu_{1}e_{3}\otimes\theta^{1},\\
\nabla e_{3}&=&\mu_{2}e_{1}\otimes \theta^{2}-\mu_{1}e_{2}\otimes
\theta^{1}\nonumber
\end{eqnarray}
where
\begin{equation*}
\mu_{i}=\frac{1}{2}(\lambda_{1}+\lambda_{2}+\lambda_{3})-\lambda_{i}, i=1,2,3.
\end{equation*}
By using relations (\ref{Eq:nabla SU(2)}), we get
\begin{eqnarray}\label{Eq: Delta e1, e2, e3}
\bar{\Delta}e_{1}&=&-\nabla_{e_{1}}\nabla_{e_{1}}e_{1}-\nabla_{e_{2}}\nabla_{e_{2}}e_{1}-\nabla_{e_{3}}\nabla_{e_{3}}e_{1}=
(\mu_{2}^{2}+\mu_{3}^{2})e_{1},\nonumber\\
\bar{\Delta}e_{2}&=&-\nabla_{e_{1}}\nabla_{e_{1}}e_{2}-\nabla_{e_{2}}\nabla_{e_{2}}e_{2}-\nabla_{e_{3}}\nabla_{e_{3}}e_{2}=
(\mu_{1}^{2}+\mu_{3}^{2})e_{2},\\
\bar{\Delta}e_{3}&=&-\nabla_{e_{1}}\nabla_{e_{1}}e_{3}-\nabla_{e_{2}}\nabla_{e_{2}}e_{3}-\nabla_{e_{3}}\nabla_{e_{3}}e_{3}=
(\mu_{1}^{2}+\mu_{2}^{2})e_{3}.\nonumber
\end{eqnarray}
Let $V=x e_{1}+y e_{2}+z e_{3}$ be a left-invariant unit vector field on $SU(2)$. By (\ref{Eq: Delta e1, e2, e3}), we obtain
\begin{eqnarray}\label{Eq: Delta V SU(2)}
\bar{\Delta}V&=& x \bar{\Delta}e_{1}+ y\bar{\Delta}e_{2}+z\bar{\Delta}e_{3}=(\mu_{2}^{2}+\mu_{3}^{2}) x e_{1}+(\mu_{1}^{2}+\mu_{3}^{2}) y e_{2}+(\mu_{1}^{2}+\mu_{2}^{2}) z e_{3},\nonumber\\
<\bar{\Delta}V, V>&=&A=(\mu_{2}^{2}+\mu_{3}^{2})x^{2}+(\mu_{1}^{2}+\mu_{3}^{2}) y^{2}+(\mu_{1}^{2}+\mu_{2}^{2}) z^{2},\\
\bar{\Delta}\bar{\Delta}V&=&(\mu_{2}^{2}+\mu_{3}^{2})^{2} x e_{1}+(\mu_{1}^{2}+\mu_{3}^{2})^{2} y e_{2}+(\mu_{1}^{2}+\mu_{2}^{2})^{2} z e_{3}\nonumber.
\end{eqnarray}
Taking into account that the expression $<\bar{\Delta}V, V>$ is a constant (since $V$ is left-invariant) and combining (\ref{Eq: biharmonic unit section}), (\ref{Eq: Delta V SU(2)}), we conclude that $V$ is a biharmonic unit section, that is $\bar{\Delta}\bar{\Delta}V-2<\bar{\Delta}V, V>\bar{\Delta}V$ collinear to $V$ if and only if there exists a real constant $\lambda$ such that
\begin{eqnarray}
x\Large\{(\mu_{2}^{2}+\mu_{3}^{2})^{2}-2(\mu_{2}^{2}+\mu_{3}^{2})A-\lambda\Large\}&=&0,\nonumber\\
y\Large\{(\mu_{1}^{2}+\mu_{3}^{2})^{2}-2(\mu_{1}^{2}+\mu_{3}^{2})A-\lambda\Large\}&=&0,\nonumber\\
z\Large\{(\mu_{1}^{2}+\mu_{2}^{2})^{2}-2(\mu_{1}^{2}+\mu_{2}^{2})A-\lambda\Large\}&=&0,\nonumber\\
x^{2}+y^{2}+z^{2}&=&1.
\end{eqnarray}
As a consequence, this yields (in combination with the Table 2 in \cite{MarkUra14}),
\begin{thm}
Let $G$ be the Lie group $SU(2)$ equipped with a left-invariant Riemannian metric $<, >$. Let $\{e_{i}, i=1, 2, 3\}$ be an orthonormal
basis of the Lie algebra satisfying (\ref{Eq:Liebracketsunimodular}) with $\lambda_{1}, \lambda_{2}, \lambda_{3}$ strictly positive constants and $\lambda_{1}\geq \lambda_{2} \geq \lambda_{3}$. We distinguish the following subcases:
\begin{itemize}
\item $\lambda_{1}=\lambda_{2}=\lambda_{3}>0$. Then, every left-invariant unit vector field is a biharmonic unit section.
\item $\lambda_{1}=\lambda_{2}>\lambda_{3}$. Then, the left-invariant vector field $V=x e_{1}+y e_{2}+z e_{3}$ is a non-harmonic biharmonic unit section  if and only if its coordinates satisfy the circle equations: $\{z=\frac{1}{\sqrt{2}}, x^{2}+y^{2}=\frac{1}{2}\}$ and $\{z=-\frac{1}{\sqrt{2}}, x^{2}+y^{2}=\frac{1}{2}\}$
\item $\lambda_{1}>\lambda_{2}=\lambda_{3}$. Then, the left-invariant vector field $V=x e_{1}+y e_{2}+z e_{3}$ is a non-harmonic biharmonic unit section  if and only if its coordinates satisfy the circle equations: $\{x=\frac{1}{\sqrt{2}}, y^{2}+z^{2}=\frac{1}{2}\}$ and $\{x=-\frac{1}{\sqrt{2}}, y^{2}+z^{2}=\frac{1}{2}\}$.
\item $\lambda_{1}>\lambda_{2}>\lambda_{3}$. In this subcase, the only non-harmonic biharmonic unit sections are:$\pm \frac{1}{\sqrt{2}} e_{2}\pm \frac{1}{\sqrt{2}} e_{3}, \pm \frac{1}{\sqrt{2}} e_{1}\pm \frac{1}{\sqrt{2}} e_{3}, \pm \frac{1}{\sqrt{2}} e_{1}\pm \frac{1}{\sqrt{2}} e_{2}$.
\end{itemize}
Furthermore, a left-invariant vector field is a biharmonic unit section if and only if is a biharmonic unit vector field.
\end{thm}
\begin{exam}
We consider the solvable Lie group $Sol_{3}$  equipped with the metric $g_{Sol}=e^{2z}dx^{2}+e^{-2z}dy^{2}+dz^{2}$, where $(x, y, z)$ are the standard coordinates on $\mathbb{R}^{3}$. We can easily check that the triple $\{e_{1}=e^{-z}\frac{\partial}{\partial x}, e_{2}=e^{z}\frac{\partial}{\partial y}, e_{3}=\frac{\partial}{\partial z}\}$ constitutes an orthonormal frame field on $Sol_{3}$.
\end{exam}
With respect to this orthonormal frame field, the Levi-Civita connection can easily be computed as (\cite{OUWang11}):
\begin{equation}\label{Eq: LeviCivita connection Sol}
\begin{array}{lcr}
\nabla_{e_{1}}e_{1}=-e_{3}, & \nabla_{e_{1}}e_{2}=0, & \nabla_{e_{1}}e_{3}=e_{1}, \\
\nabla_{e_{2}}e_{1}=0, & \nabla_{e_{2}}e_{2}=e_{3}, & \nabla_{e_{2}}e_{3}=-e_{2},\\
\nabla_{e_{3}}e_{1}=0, & \nabla_{e_{3}}e_{2}=0, & \nabla_{e_{3}}e_{3}=0.\\
\end{array}
\end{equation}
Let $V=x e_{1}+y e_{2}+z e_{3}$ be a left-invariant unit vector field on $Sol_{3}$. By using (\ref{Eq: LeviCivita connection Sol}), we calculate
\begin{equation*}
\begin{array}{lcr}
\bar{\Delta}V=x e_{1}+y e_{2}+2z e_{3}, & g_{Sol}(\bar{\Delta}V, V)=x^{2}+y^{2}+2z^{2}=1+z^{2}, \\
\bar{\Delta}\bar{\Delta}V=x e_{1}+y e_{2}+4z e_{3}.
\end{array}
\end{equation*}
As a consequence, $V$ is a harmonic unit section, equivalently, $\bar{\Delta}V=g_{Sol}(\bar{\Delta}V, V)V$ if and only if
\begin{equation*}
\begin{array}{lccr}
x z^{2}=0,  & y z^{2}=0, & z(1-z^{2})=0, & x^{2}+y^{2}+z^{2}=1.
\end{array}
\end{equation*}
$V$ is a biharmonic unit section if and only if there exists a constant $\lambda$ such that
\begin{equation*}
\begin{array}{lccr}
x(1+2z^{2}+\lambda)=0, & y(1+2z^{2}+\lambda)=0, & z(\lambda+4z^{2})=0, & x^{2}+y^{2}+z^{2}=1.
\end{array}
\end{equation*}
Summarizing, we obtain
\begin{thm}
Let $V=x e_{1}+y e_{2}+z e_{3}$ be a left-invariant unit vector field on $Sol_{3}$. Then,
\begin{enumerate}
\item $V$ is a harmonic unit section if and only if either $V=\pm e_{3}$ or $V\in e_{3}^{\perp}$.
\item $V$ is a non-harmonic biharmonic unit section if and only if its coordinates satisfy the circle equations: $\{z=\frac{1}{\sqrt{2}}, x^{2}+y^{2}=\frac{1}{2}\}$ and $\{z=-\frac{1}{\sqrt{2}}, x^{2}+y^{2}=\frac{1}{2}\}$.
\end{enumerate}
\end{thm}
\begin{exam}
Let $(H^{n}, g)$ be the $n$-dimensional Poincar\'{e} half space, i.e.
\begin{equation*}
H^{n}=\{(y_{1}, y_{2}, \ldots, y_{n})\in \Real^{n}|y_{1}>0\}
\end{equation*}
and $g$ is the Riemannian metric given by
\begin{equation*}
g=(c y_{1})^{-2}\sum_{i=1}^{n}(dy_{i})^{2}.
\end{equation*}
\end{exam}
It has constant curvature $-c^{2}$. By using the theory of homogeneous structures, we use the representation of $H^{n}$ as the subgroup $G$ of $GL(n, \Real)$ of the form (\cite{MedrGonVanh01})
\begin{equation*}
\begin{array}{c}
a
\end{array}
=\left (
\begin{array}{ccccc}
e^{x_{1}} & 0 & \ldots  & 0  & x_{2}\\
0 & e^{x_{1}} & \ldots & 0 &  x_{3}\\
. & . & .  & . & .\\
. & . & . & . & . \\
0 & . & . & e^{x_{1}} & x_{n}\\
0 & . & . & 0 & 1
\end{array}\right )
\end{equation*}
It is a solvable Lie group which is a semi-direct product of the multiplicative $\Real_{0}^{+}=\{x\in\Real|x>0\}$ and the additive group $\Real^{n-1}$. The left-invariant vector fields
\begin{equation*}
\begin{array}{lcr}
X_{1}=c\frac{\partial}{\partial x_{1}},  & X_{i}=c e^{x_{1}}\frac{\partial}{\partial x_{i}},  & (2\leq i \leq n)
\end{array}
\end{equation*}
constitute a basis of the Lie algebra. On $G$, we consider the Riemannian metric $< , >$ for which these vector fields form an orthonormal basis at each point. Then, $G$ is isometric to the Poincar\'{e} half-space $(H^{n}, g)$. The Levi-Civita connection with respect to this metric is determined by
\begin{equation}\label{Eq: LeviCivita connection Hn}
\begin{array}{lcr}
\nabla_{X_{i}}X_{i}=cX_{1}, & \nabla_{X_{i}}X_{1}=-cX_{i}  & (2\leq i \leq n)
\end{array}
\end{equation}
and the remaining covariant derivatives vanish. Let $V=V_{1}X_{1}+\sum_{i=2}^{n}V_{i}X_{i}$ be a left-invariant unit vector field on $G$. Using (\ref{Eq: LeviCivita connection Hn}), we get (\cite{MedrGonVanh01})
\begin{equation}\label{Eq: div and Laplacian Hn}
\begin{array}{lccr}
\nabla_{X_{1}}V=0, & \nabla_{X_{i}}V=-c V_{1} X_{i}+c V_{i} X_{1}, & \bar{\Delta}X_{1}=c^{2}(n-1)X_{1}, & \bar{\Delta}X_{i}=c^{2} X_{i},
\end{array}
\end{equation}
for $2\leq i \leq n$. As a consequence, we obtain
\begin{eqnarray}\label{Eq: S(V) and DV Hn}
\bar{\Delta}V&=&V_{1}\bar{\Delta}X_{1}+\sum_{i=2}^{n}V_{i}\bar{\Delta}X_{i}=c^{2}(n-1)V_{1}X_{1}+c^{2}\sum_{i=2}^{n}V_{i}X_{i},\nonumber\\
<\bar{\Delta}V, V>&=&c^{2}(n-1)V_{1}^{2}+c^{2}\sum_{i=2}^{n}V_{i}^{2}=c^{2}[(n-1)V_{1}^{2}+1-V_{1}^{2}]=c^{2}[1+(n-2)V_{1}^{2}],\nonumber\\
\bar{\Delta}\bar{\Delta}V&=&c^{4}(n-1)^{2}V_{1}X_{1}+c^{4}\sum_{i=2}^{n}V_{i}X_{i},\nonumber\\
S(V)&=&\sum_{i=1}^{n}R(\nabla_{X_{i}}V, V)X_{i}=(-c^{2})\sum_{i=2}^{n}\Big\{<V, X_{i}>\nabla_{X_{i}}V-<\nabla_{X_{i}}V, X_{i}>V\Big\}\nonumber\\
&&=(-c^{3})\Big\{[(n-2)V_{1}^{2}+1]X_{1}+(n-2)V_{1}\sum_{i=2}^{n}V_{i}X_{i}\Big\},
\end{eqnarray}
where we have used that $(H^{n}, < , >)$ is a space of constant curvature $-c^{2}$ and $V_{1}^{2}+\sum_{i=2}^{n}V_{i}^{2}=1$. For $n=2$ every left-invariant unit vector field on $(H^{n}, < , >)$ is harmonic unit section and for $n>2$, $V$ is a harmonic unit section if and only if either $V=\pm X_{1}$ or $V\in \{X_{1}\}^{\perp}$, where $X_{1}^{\perp}$ denotes the distribution defined by the vector fields orthogonal to $X_{1}$ (\cite{MedrGonVanh01}). Combining (\ref{Eq: div and Laplacian Hn}) and (\ref{Eq: S(V) and DV Hn}), we calculate
\begin{equation*}
\begin{array}{lcr}
\nabla_{X_{1}}S(V)=0, & \nabla_{X_{i}}S(V)=(-c^{4})\Big\{(n-2)V_{1}V_{i}X_{1}-[(n-2)V_{1}^{2}+1]X_{i}\Big\}, & (2\leq i \leq n)
\end{array}
\end{equation*}
and
\begin{eqnarray}\label{Eq: curvature term biharmonic unit vector}
\sum_{i=2}^{n}R(X_{i}, \nabla_{X_{i}}S(V))V&=&(-c^{4})(n-2)\sum_{i=2}^{n}V_{1}V_{i}R(X_{i}, X_{1})V\nonumber\\
&=&c^{6}(n-2)\sum_{i=2}^{n}V_{1}V_{i}\Big\{<V, X_{1}>X_{i}-<V, X_{i}>X_{1}\Big\}\nonumber\\
&=&c^{6}(n-2)\Big\{-V_{1}(1-V_{1}^{2})X_{1}+V_{1}^{2}\sum_{i=2}^{n}V_{i}X_{i}\Big\},\nonumber\\
\sum_{i=2}^{n}R(X_{i}, S(V))\nabla_{X_{i}}V&=&c\sum_{i=2}^{n}V_{i}R(X_{i}, S(V))X_{1}-cV_{1}\sum_{i=2}^{n}R(X_{i}, S(V))X_{i}\nonumber\\
&=&c^{6}\Big\{(n-1)+(n-1)(n-2)V_{1}^{2}\Big\}V_{1}X_{1}+c^{6}\sum_{i=2}^{n}\Big\{1+(n-1)(n-2)V_{1}^{2}\Big\}V_{i}X_{i}.
\end{eqnarray}
By using (\ref{Eq: biharmonic unit section}), (\ref{Eq: biharmonic unit vector field}), (\ref{Eq: S(V) and DV Hn}) and (\ref{Eq: curvature term biharmonic unit vector}), we conclude that $V$ is a biharmonic unit section if and only if there exists a constant $\lambda$ such that
\begin{equation*}
\begin{array}{lcr}
c^{4}\Big\{(n-1)(n-3)-2(n-1)(n-2)V_{1}^{2}\Big\}V_{1}=\lambda V_{1}, & -c^{4}\Big\{1+2(n-2)V_{1}^{2}\Big\}V_{i}=\lambda V_{i}, & (2\leq i \leq n)
\end{array}
\end{equation*}
Furthermore, $V$ is a biharmonic unit vector field if and only if there exists a constant $\mu$ such that
\begin{eqnarray}
c^{4}\Big\{(n-1)(n-3)-2(n-1)(n-2)V_{1}^{2}\Big\}V_{1}+c^{6}\Big\{n+(n-2)(2n-1)V_{1}^{2}\Big\}V_{1}&=&\mu V_{1},\nonumber\\
-c^{4}\Big\{1+2(n-2)V_{1}^{2}\Big\}V_{i}+c^{6}\Big\{2+(n-2)(2n-1)V_{1}^{2}\Big\}V_{i}&=&\mu V_{i},\nonumber
\end{eqnarray}
for $2\leq i \leq n$.
Summarizing, we yield
\begin{thm}
Let $H^{n}\cong G$ be the $n$-dimensional Poincar\'{e} half-space. For $n=2$, every left-invariant unit vector field is biharmonic unit section and biharmonic unit vector field. For $n>2$, we have
\begin{enumerate}
\item The left-invariant unit vector field $V=V_{1}X_{1}+\sum_{i=2}^{n}V_{i}X_{i}$ is a non-harmonic biharmonic unit section if and only if its coordinates satisfy the equations of the $(n-2)$-hyperspheres of $\Real^{n}$: $C_{1}=\{V_{1}=\frac{1}{\sqrt{2}}, \sum_{i=2}^{n}V_{i}^{2}=\frac{1}{2}\}$ and $C_{2}=\{V_{1}=-\frac{1}{\sqrt{2}}, \sum_{i=2}^{n}V_{i}^{2}=\frac{1}{2}\}$.
\item The set of left-invariant biharmonic unit vector field is: $V=\pm X_{1}$, $V\in X_{1}^{\perp}$ and the equations of the $(n-2)$-hyperspheres of $\Real^{n}$: $C_{3}=\{V_{1}=\sqrt{\frac{c^{2}+n-2}{2(n-2)}}, \sum_{i=2}^{n}V_{i}^{2}=\frac{n-2-c^{2}}{2(n-2)}\}$ and $C_{4}=\{V_{1}=-\sqrt{\frac{c^{2}+n-2}{2(n-2)}}, \sum_{i=2}^{n}V_{i}^{2}=\frac{n-2-c^{2}}{2(n-2)}\}$ with the assumption $c^{2}<n-2$, where $X_{1}^{\perp}$ denotes the distribution defined by the vector fields orthogonal to $X_{1}$.
\end{enumerate}
\end{thm}
\begin{rem}
The notions ''harmonic unit section'' and ''harmonic unit vector fields'' are equivalent. On the other hand, the left-invariant unit vector fields of $H^{n}$ whose coordinates satisfy the equations of the hyperspheres $C_{1}$ and $C_{2}$ (resp. $C_{3}$ and $C_{4}$) are biharmonic unit sections (resp. biharmonic unit vector fields) which are not biharmonic unit vector fields (resp. biharmonic unit sections).
\end{rem}
\begin{rem}
Theorem 3.3 completes and corrects the statement of \cite[Theorem 3.1]{Mark09}, which is based on the incorrect \cite[Proposition 3.1]{Mark09}.
\end{rem}

\section{Conformal change of the metric and the tension field of a vector field}

In this section we derive the critical points of the bienergy functional (resp. vertical bienergy functional) for unit vector fields after a conformal change of the metric on the base manifold. We have
\begin{thm}
Let $(M^{n}, g)$ be a Riemannian manifold, $V$ a unit vector field on $M$ and $u \in C^{\infty}(M)$. We consider the conformal change of metric $\widetilde{g}=e^{2u}g$ and $\widetilde{V}=e^{-u}V$. Then,
\begin{eqnarray}\label{Eq: conformal change Delta V}
\widetilde{\bar{\Delta}}\widetilde{V}-\widetilde{g}(\widetilde{\bar{\Delta}}\widetilde{V}, \widetilde{V})\widetilde{V}&=&e^{-3u}\Big\{\bar{\Delta}V-g(\bar{\Delta}V, V)V+(2-n)\nabla_{\grad u}V-2\sum_{i=1}^{n-1}g(\grad u, \nabla_{E_{i}}V)E_{i}\nonumber\\
&&+[2\di V+(n-2)V(u)]\sum_{i=1}^{n-1}E_{i}(u)E_{i}\Big\}, \\
\hspace{2.5cm}\widetilde{\nabla}_{\widetilde{V}}\widetilde{V}-(\widetilde{\di} \widetilde{V})\tilde{V}&=&e^{-2u}\Big\{\nabla_{V}V-(\di V) V+(2-n)V(u)V - \grad u\Big\},\label{Eq: conformal change S(V)}
\end{eqnarray}
where $\{E_{1}, E_{2}, \ldots, E_{n-1}, E_{n}=V\}$ is an orthonormal frame field of $(M, g)$.
\end{thm}
\begin{proof}
We consider an arbitrary orthonormal frame field $\{E_{1}, E_{2}, \ldots, E_{n}\}$ of $(M, g)$. We set $\widetilde{E}_{i}=e^{-u}E_{i}, (1\leq i \leq n)$. Then, $\{\widetilde{E}_{1}, \widetilde{E}_{2}, \ldots, \widetilde{E}_{n}\}$ is an orthonormal frame field of $(M, \widetilde{g})$. The Levi-Civita connections $\nabla$ and $\widetilde{\nabla}$ of $(M, g)$ and $(M, \widetilde{g})$ are related by (\cite[page 107]{DragomirPerro11}):
\begin{equation}\label{Eq:Levicivitaconformal}
\widetilde{\nabla}_{Y}{Z}=\nabla_{Y}Z+Y(u)Z+Z(u)Y-g(Y, Z)\grad u
\end{equation}
for every $Y, Z\in \X(M)$. Therefore
\begin{equation}\label{Eq:tildenablaV}
\widetilde{\nabla}_{\widetilde{E}_{i}}\widetilde{V}=e^{-2u}\Big\{\nabla_{E_{i}}V+V(u)E_{i}-g(V, E_{i})\grad u\Big\}.
\end{equation}
By using (\ref{Eq:Levicivitaconformal}) and (\ref{Eq:tildenablaV}), we get
\begin{eqnarray}
\sum_{i=1}^{n}\widetilde{\nabla}_{\widetilde{E_{i}}}\widetilde{\nabla}_{\widetilde{E_{i}}}\widetilde{V}&=&e^{-3u}\Big\{-\nabla_{\grad u}V+\nabla_{E_{i}}\nabla_{E_{i}}V-2\di V\grad u + \sum_{i=1}^{n}(\nabla_{E_{i}}V)(u)E_{i}+\sum_{i=1}^{n}E_{i}[V(u)]E_{i}\nonumber\\
&&+V(u)\sum_{i=1}^{n}\nabla_{E_{i}}E_{i}+(2-n)V(u)\sum_{i=1}^{n}E_{i}(u)E_{i}-\sum_{i=1}^{n}g(V, \nabla_{E_{i}}E_{i})\grad u\nonumber\\
&&-\nabla_{V}\grad u-\|\grad u\|^{2}V\Big\},\nonumber\\
\sum_{i=1}^{n}\widetilde{\nabla}_{\widetilde{E}_{i}}\widetilde{E}_{i}&=&e^{-2u}\Big\{(1-n)\grad u+\sum_{i=1}^{n}\nabla_{E_{i}}E_{i}\Big\},\nonumber\\
\sum_{i=1}^{n}\widetilde{\nabla}_{\widetilde{\nabla}_{\widetilde{E}_{i}}\widetilde{E}_{i}}\widetilde{V}&=&e^{-3u}\Big\{(1-n)\nabla_{\grad u}V+\sum_{i=1}^{n}\nabla_{\nabla_{E_{i}}E_{i}}V+V(u)\sum_{i=1}^{n}\nabla_{E_{i}}E_{i}-\sum_{i=1}^{n}g(\nabla_{E_{i}}E_{i}, V)\grad u\Big\}.\nonumber
\end{eqnarray}
As a consequence, we obtain
\begin{eqnarray}\label{eq: tilde Delta V}
\widetilde{\bar{\Delta}}\widetilde{V}&=&\sum_{i=1}^{n}\widetilde{\nabla}_{\widetilde{\nabla}_{\widetilde{E}_{i}}\widetilde{E}_{i}}\widetilde{V}-
\sum_{i=1}^{n}\widetilde{\nabla}_{\widetilde{E_{i}}}\widetilde{\nabla}_{\widetilde{E_{i}}}\widetilde{V}\nonumber\\
&=&e^{-3u}\Big\{\bar{\Delta}V+(2-n)\nabla_{\grad u}V+\nabla_{V}\grad u+2\di V \grad u-\sum_{i=1}^{n}(\nabla_{E_{i}}V)(u)E_{i}\nonumber\\
&&+(n-2)V(u)\grad u-\sum_{i=1}^{n}E_{i}[V(u)]E_{i}+\|\grad u\|^{2}V\Big\},\nonumber\\
\widetilde{g}(\widetilde{\bar{\Delta}}\widetilde{V}, \widetilde{V})&=&e^{-2u}\Big\{g(\bar{\Delta}V, V)+g(\nabla_{V}\grad u, V)+2\di V V(u)-(\nabla_{V}V)(u)+(n-2)V(u)^{2}\nonumber\\
&&-V V(u)+\|\grad u\|^{2}\Big\}.
\end{eqnarray}
In the sequel, we consider $E_{n}=V$. Using (\ref{eq: tilde Delta V}) and taking into account $E_{i}[V(u)]=E_{i}g(V, \grad u)=g(\grad u, \nabla_{E_{i}}V)+g(\nabla_{V}\grad u, E_{i})$, we get (\ref{Eq: conformal change Delta V}). By using (\ref{Eq:Levicivitaconformal}) and (\ref{Eq:tildenablaV}), we get
\begin{eqnarray}
\widetilde{\nabla}_{\widetilde{V}}\widetilde{V}&=&\nabla_{\widetilde{V}}\widetilde{V}+2\widetilde{V}(u)\widetilde{V}-g(\widetilde{V}, \widetilde{V})\grad u=e^{-2u}\Big\{\nabla_{V}V+V(u)V-\grad u\Big\},\nonumber\\
\widetilde{\di} \widetilde{V}&=&\sum_{i=1}^{n-1}\widetilde{g}(\widetilde{\nabla}_{\widetilde{E}_{i}}\widetilde{V}, \widetilde{E}_{i})=e^{-u}\Big\{\di V+(n-1)V(u)\Big\},\nonumber
\end{eqnarray}
from which (\ref{Eq: conformal change S(V)}) is easily deduced.
\end{proof}
\begin{cor}
Let $(M^{2}, g)$ be a 2-dimensional Riemannian manifold, $V$ a unit vector field on $M$ and $u \in C^{\infty}(M)$. We consider the conformal change of metric $\widetilde{g}=e^{2u}g$ and $\widetilde{V}=e^{-u}V$. Then,
\begin{equation}\label{Eq: tilde Delta V 2 dim}
\widetilde{\bar{\Delta}}\widetilde{V}-\widetilde{g}(\widetilde{\bar{\Delta}}\widetilde{V}, \widetilde{V})\widetilde{V}= e^{-3u}\Big\{\bar{\Delta}V-g(\bar{\Delta}V, V)V\Big\}.
\end{equation}
As a consequence, $V$ is a harmonic unit section on $(M^{2}, g)$ if and only if $\widetilde{V}$ is a harmonic unit section on $(M^{2}, \widetilde{g})$.
\end{cor}
\begin{proof}
We consider the local orthonormal frame field $\{E_{1}, V\}$. The Levi-Civita connection of $(M^{2}, g)$ is given by
\begin{equation}\label{Eq: connection M2 V}
\begin{array}{lcr}
\nabla_{V}V=a E_{1}, & \nabla_{E_{1}}E_{1}=-b V,\\
\nabla_{E_{1}}V=b E_{1}, & \nabla_{V}E_{1}=-a V, & [V, E_{1}]=-a V-b E_{1},
\end{array}
\end{equation}
for some functions $a, b$ on $M$. We easily notice that $\di V=b$ and $\di E_{1}=-a$. By using (\ref{Eq: conformal change Delta V}) and (\ref{Eq: connection M2 V}), we get (\ref{Eq: tilde Delta V 2 dim}). The second part of the Corollary follows from by the definition of harmonic unit sections.
\end{proof}
\begin{rem}
We should point out that conformal invariance of harmonic unit sections (Corollary 4.2) is of a different nature from the conformal invariance of harmonic maps on surfaces.
\end{rem}
We assume now that $(M, g)$ is a 2-dimensional compact Riemannian manifold. By Corollary 4.2 and $v_{\widetilde{g}}=e^{2u}v_{g}$(\cite[page 107]{DragomirPerro11}), we get
\begin{equation*}
E_{2}^{v}(V)=\frac{1}{2}\int_{M}\|\bar{\Delta}V-g(\bar{\Delta}V, V)V\|_{g}^{2}v_{g}=\frac{1}{2}\int_{M}e^{2u}\|\widetilde{\bar{\Delta}}\widetilde{V}-\widetilde{g}(\widetilde{\bar{\Delta}}\widetilde{V}, \widetilde{V})\widetilde{V}\|_{\widetilde{g}}^{2}v_{\widetilde{g}}.
\end{equation*}
In the sequel, we derive a condition on $\tilde{V}$ for $V$ to be a biharmonic unit section. In particular, we have
\begin{thm}
Let $(M^{2}, g)$ be a 2-dimensional compact Riemannian manifold, $V$ a unit vector field on $M$ and $u \in C^{\infty}(M)$. We consider the conformal change of metric $\widetilde{g}=e^{2u}g$ and $\widetilde{V}=e^{-u}V$. Then, $V$ is a biharmonic unit section if and only if,
\begin{equation}\label{Eq: conformal change biharmonic unit section}
\widetilde{\bar{\Delta}}\Big(e^{2u}
(\widetilde{\bar{\Delta}}\widetilde{V}-\widetilde{g}(\widetilde{\bar{\Delta}}\widetilde{V}, \widetilde{V})\widetilde{V})\Big)-e^{2u}\widetilde{g}(\widetilde{\bar{\Delta}}\widetilde{V}, \widetilde{V})\widetilde{\bar{\Delta}}\widetilde{V}
\end{equation}
is collinear to $V$.
\end{thm}
\begin{proof}
Recall that $V$ is a biharmonic unit section if and only if $\frac{d}{dt}E_{2}^{v}(V_{t})|_{t=0}=0$ for all variations $\{V_{t}=e^{u}\widetilde{V}_{t}\}_{t\in I}$ within $\X^{1}(M)$. We adopt the terminology of \cite{MarkUra14} and \cite{MarkUra}.
We consider the function
\begin{equation*}
E_{2}^{v}(\widetilde{V}_{t})=E_{2}^{v}(t)=\frac{1}{2}\int_{M}e^{2u}\Big\{\widetilde{g}(\widetilde{\bar{\Delta}}\widetilde{V}_{t}, \widetilde{\bar{\Delta}}\widetilde{V_{t}})-\widetilde{g}(\widetilde{\bar{\Delta}}\widetilde{V}_{t}, \widetilde{V}_{t})^{2}\Big\}v_{\widetilde{g}}.
\end{equation*}
Differentiating the function $E_{2}^{v}(t)$ w.r.t. the variable $t$, we obtain
\begin{eqnarray}\label{Eq: derivative vertical bienergy conformal}
\frac{d}{dt}E_{2}^{v}(t)&=&\int_{M}e^{2u}\widetilde{g}(\widetilde{\nabla}_{\partial_{t}}\widetilde{\bar{\Delta}}\widetilde{V}_{t}, \widetilde{\bar{\Delta}}\widetilde{V}_{t})v_{\widetilde{g}}-\int_{M}e^{2u}\widetilde{g}(\widetilde{\nabla}_{\partial_{t}}\widetilde{\bar{\Delta}}\widetilde{V}_{t}, \widetilde{V}_{t})\widetilde{g}(\widetilde{\bar{\Delta}}\widetilde{V}_{t}, \widetilde{V}_{t})v_{\widetilde{g}}\nonumber\\
&&-\int_{M}e^{2u}\widetilde{g}(\widetilde{\bar{\Delta}}\widetilde{V}_{t}, \widetilde{\nabla}_{\partial_{t}}\widetilde{V}_{t})\widetilde{g}(\widetilde{\bar{\Delta}}\widetilde{V}_{t}, \widetilde{V}_{t})v_{\widetilde{g}}.
\end{eqnarray}
For the sake of convenience, we omit the tilde. From the proof of Theorem 3.2 in \cite{MarkUra14}, we have the following formulas,
\begin{eqnarray}
g(\nabla_{\partial_{t}}\bar{\Delta}V_{t}, \bar{\Delta}V_{t})&=&\Delta[g(\nabla_{\partial_{t}}V_{t}, \bar{\Delta}V_{t})]+2\di \theta_{t}+g(\nabla_{\partial_{t}}V_{t}, \bar{\Delta}\bar{\Delta}V_{t}),\nonumber\\
g(\nabla_{\partial_{t}}\bar{\Delta}V_{t}, V_{t})&=&2\di \omega_{t}+g(\nabla_{\partial_{t}}V_{t}, \bar{\Delta}V_{t}),\nonumber
\end{eqnarray}
where $\theta_{t}(\cdot)=g(\nabla_{\partial_{t}}V_{t}, \nabla_{\cdot}\bar{\Delta}V_{t}), t \in I$ and $\omega_{t}(\cdot)=g(\nabla_{\partial_{t}}V_{t}, \nabla_{\cdot}V_{t})$, $t\in I$. Therefore, (\ref{Eq: derivative vertical bienergy conformal}) gives
\begin{eqnarray}
\frac{d}{dt}E_{2}^{v}(t)&=&\int_{M}e^{2u}\Delta[g(\nabla_{\partial_{t}}V_{t}, \bar{\Delta}V_{t})]v_{g}+2\int_{M}e^{2u}\di \theta _{t}+\int_{M}e^{2u}g(\bar{\Delta}\bar{\Delta}V_{t}, \nabla_{\partial_{t}}V_{t})\nonumber\\
&&-2\int_{M}e^{2u}\di \omega_{t}g(\bar{\Delta}V_{t}, V_{t})v_{g}-2\int_{M}e^{2u}g(\bar{\Delta}V_{t}, V_{t})g(\bar{\Delta}V_{t}, \nabla_{\partial_{t}}V_{t})v_{g}\nonumber\\
&=&\int_{M}g\Big(\Delta(e^{2u})\bar{\Delta}V_{t}-2\nabla_{\grad(e^{2u})}\bar{\Delta}V_{t}+e^{2u}\bar{\Delta}\bar{\Delta}V_{t}+
2\nabla_{\grad[e^{2u}g(\bar{\Delta}V_{t}, V_{t})]}V_{t}\nonumber\\
&&-2e^{2u}g(\bar{\Delta}V_{t}, V_{t})\bar{\Delta}V_{t}, \nabla_{\partial_{t}}V_{t}\Big)v_{g}.\nonumber
\end{eqnarray}
By the formula $\bar{\Delta}(fX)=(\Delta f)X+f\bar{\Delta}X-2\nabla_{\grad f}X$, we deduce
\begin{equation*}
\frac{d}{dt}|_{t=0}E_{2}^{v}(t)=\int_{M}\widetilde{g}\Big(X,\widetilde{\bar{\Delta}}\Big(e^{2u}
(\widetilde{\bar{\Delta}}\widetilde{V}-\widetilde{g}(\widetilde{\bar{\Delta}}\widetilde{V}, \widetilde{V})\widetilde{V})\Big)-e^{2u}\widetilde{g}(\widetilde{\bar{\Delta}}\widetilde{V}, \widetilde{V})\widetilde{\bar{\Delta}}\widetilde{V}\Big)v_{\widetilde{g}},
\end{equation*}
where $X=\nabla_{\partial_{t}}V_{t}|_{t=0}$.
\end{proof}
For a 2-dimensional manifold $(M, g)$ with Gaussian curvature $k_{g}$ and $V$ a unit vector field on $M$, we have $S(V)=\sum_{i=1}^{2}R(\nabla_{e_{i}}V, V)e_{i}=k_{g}(\nabla_{V}V-\di (V )V)$. By (\ref{Eq: conformal change S(V)}), with $n=2$, we have
\begin{eqnarray}
E_{2}(V)&=&\frac{1}{2}\int_{M}\Big\{\|\bar{\Delta}V-g(\bar{\Delta}V, V)V\|^{2}+\|S(V)\|^{2}\Big\}v_{g}\nonumber\\
&=&\frac{1}{2}\int_{M}\Big\{\|\bar{\Delta}V-g(\bar{\Delta}V, V)V\|^{2}+k_{g}^{2}\|\nabla_{V}V-\di (V) V\|^{2}\Big\}v_{g}\nonumber\\
&=&\frac{1}{2}\int_{M}\Big\{e^{2u}\|\widetilde{\bar{\Delta}}\widetilde{V}-\widetilde{g}(\widetilde{\bar{\Delta}}\widetilde{V}, \widetilde{V})\widetilde{V}\|_{\widetilde{g}}^{2}+k_{g}^{2}\Big[\|\widetilde{\nabla}_{\widetilde{V}}\widetilde{V}-\widetilde{\di}(\widetilde{V})\widetilde{V}\|_{\widetilde{g}}^{2}
+\|\widetilde{\grad}u\|_{\widetilde{g}}^{2}\nonumber\\
&&+2\widetilde{g}(\widetilde{\nabla}_{\widetilde{V}}\widetilde{V}-\widetilde{\di}(\widetilde{V})\widetilde{V}, \widetilde{\grad}u)\Big]\Big\}v_{\widetilde{g}}.\nonumber
\end{eqnarray}
In the sequel, we derive the condition on $\widetilde{V}$ for $V$ to be a biharmonic unit vector field. It is enough to deal with the derivative of horizontal part of the integral since the derivative of the vertical part was obtained in Theorem 4.3.
\begin{thm}
Let $(M^{2}, g)$ be a 2-dimensional compact Riemannian manifold, $V$ a unit vector field on $M$ and $u \in C^{\infty}(M)$. We consider the conformal change of metric $\widetilde{g}=e^{2u}g$ such that $(M, \widetilde{g})$ is flat and $\widetilde{V}=e^{-u}V$. Then, $V$ is a biharmonic unit vector field if and only if,
\begin{eqnarray}\label{Eq: conformal change biharmonic unit vector field}
\tilde{\bar{\Delta}}\Big(e^{2u}
(\tilde{\bar{\Delta}}\tilde{V}-\tilde{g}(\tilde{\bar{\Delta}}\tilde{V}, \tilde{V})\tilde{V})\Big)-e^{2u}\tilde{g}(\tilde{\bar{\Delta}}\tilde{V}, \tilde{V})\tilde{\bar{\Delta}}\tilde{V}+[\tilde{V}, J\tilde{V}]k_{g}^{2}J\tilde{V}+k_{g}^{2}(\tilde{\bar{\Delta}}\tilde{V}-\tilde{g}(\tilde{\bar{\Delta}}\tilde{V}, \tilde{V})\tilde{V})+\nonumber\\
+\tilde{g}(J\widetilde{\grad}u, \widetilde{\grad}k_{g}^{2})J\tilde{V}
\end{eqnarray}
is collinear to $\widetilde{V}$, where $J$ is the natural complex structure defined on $(M^{2}, g)$ and $\grad u=e^{2u}\widetilde{\grad} u$.
\end{thm}
\begin{proof}
Using the orthonormal frame field $\{\widetilde{E}_{1}=J(\widetilde{V}), \widetilde{V}\}$ of $(M, \widetilde{g})$  satisfying $(\ref{Eq: connection M2 V})$, we compute
\begin{eqnarray}
\widetilde{\bar{\Delta}}V&=&\widetilde{\nabla}_{\widetilde{\nabla}_{\widetilde{E}_{1}}\widetilde{E}_{1}}\widetilde{V}-\widetilde{\nabla}_{\widetilde{E}_{1}}\widetilde{\nabla}_
{\widetilde{E}_{1}}\widetilde{V}+\widetilde{\nabla}_{\widetilde{\nabla}_{\widetilde{V}}\widetilde{V}}\widetilde{V}-
\widetilde{\nabla}_{\widetilde{V}}\widetilde{\nabla}_{\widetilde{V}}\widetilde{V}\nonumber\\
&=&-b\widetilde{\nabla}_{\widetilde{V}}\widetilde{V}-\widetilde{E}_{1}(b)\widetilde{E}_{1}-
b\widetilde{\nabla}_{\widetilde{E}_{1}}\widetilde{E}_{1}+a\widetilde{\nabla}_{\widetilde{E}_{1}}\widetilde{V}-\widetilde{V}(a)\widetilde{E}_{1}-a\widetilde{\nabla}_{\widetilde{V}}\widetilde{E}_{1}\nonumber\\
&=&-\widetilde{E}_{1}(b)\widetilde{E}_{1}-\widetilde{V}(a)\widetilde{E}_{1}+(a^{2}+b^{2})\widetilde{V}.\nonumber
\end{eqnarray}
As a consequence,
\begin{equation}\label{Eq: Laplacian 2 manifold product JV}
\widetilde{g}(\widetilde{\bar{\Delta}}\widetilde{V}-\widetilde{g}(\widetilde{\bar{\Delta}}\widetilde{V}, \widetilde{V})\widetilde{V}, J(\widetilde{V}))=-\widetilde{E}_{1}(b)-\widetilde{V}(a)=\widetilde{V}(\widetilde{\di} J(\widetilde{V}))-J(\widetilde{V})(\widetilde{\di}\widetilde{V}).
\end{equation}
We consider the horizontal part of the bienergy functional in terms of $\widetilde{V}$:
\begin{equation*}
F(\widetilde{V}_{t})=F(t)=\frac{1}{2}\int_{M}\Big\{k_{g}^{2}\|\widetilde{\nabla}_{\widetilde{V}_{t}}\widetilde{V}_{t}-(\widetilde{\di} \widetilde{V}_{t})\widetilde{V}_{t}\|_{\widetilde{g}}^{2}+2k_{g}^{2}\widetilde{g}(\widetilde{\nabla}_{\widetilde{V}_{t}}\widetilde{V}_{t}-(\widetilde{\di}\widetilde{V}_{t})\widetilde{V}_{t}, \widetilde{\grad}u)+k_{g}^{2}\|\widetilde{\grad}u\|_{\widetilde{g}}^{2}\Big\}v_{\widetilde{g}}.
\end{equation*}
Differentiating the function $F(t)$ w.r.t. the variable $t$, we obtain
\begin{eqnarray}\label{Eq: derivative F}
\frac{d}{dt}F(t)&=&\frac{d}{dt}\frac{1}{2}\int_{M}k_{g}^{2}\Big(\widetilde{g}(\widetilde{\nabla}_{\widetilde{V}_{t}}\widetilde{V}_{t}, \widetilde{\nabla}_{\widetilde{V}_{t}}\widetilde{V}_{t})+(\widetilde{\di} \widetilde{V}_{t})^{2}\Big) v_{\widetilde{g}}\nonumber\\
&&+\frac{d}{dt}\int_{M}k_{g}^{2}\widetilde{g}\Big(\widetilde{\nabla}_{\widetilde{V}_{t}}\widetilde{V}_{t}-(\widetilde{\di} \widetilde{V}_{t})\widetilde{V}_{t}, \widetilde{\grad} u\Big)v_{\widetilde{g}}
\end{eqnarray}
In the sequel, we omit the tilde for simplification. Assuming that $(M, \widetilde{g})$ is flat, we get
\begin{eqnarray}
\frac{d}{dt}&\frac{1}{2}\int_{M}&k_{g}^{2}\Big(g(\nabla_{V_{t}}V_{t}, \nabla_{V_{t}}V_{t})+(\di V_{t})^{2}\Big)v_{g}=\int_{M}\Big\{k_{g}^{2}g(\nabla_{\partial_{t}}\nabla_{V_{t}}V_{t}, \nabla_{V_{t}}V_{t})+k_{g}^{2}(\di V_{t})\frac{d}{dt}[\di V_{t}]\Big\}v_{g}\nonumber\\
&=&\int_{M}\Big(k_{g}^{2}g(\nabla_{V_{t}}\nabla_{\partial_{t}}V_{t}+\nabla_{[\partial_{t}, V_{t}]}V_{t}, \nabla_{V_{t}}V_{t})+k_{g}^{2}(\di V_{t})g(\nabla_{\partial_{t}}\nabla_{e_{i}}V_{t}, e_{i})\Big)v_{g}\nonumber\\
&=&\int_{M}\Big(k_{g}^{2}g(\nabla_{V_{t}}\nabla_{\partial_{t}}V_{t}+\nabla_{\nabla_{\partial_{t}}V_{t}}V_{t}, \nabla_{V_{t}}V_{t})+k_{g}^{2}(\di V_{t})g(\nabla_{e_{i}}\nabla_{\partial_{t}}V_{t}, e_{i})\Big)v_{g}\nonumber\\
&=&\int_{M}\Big(k_{g}^{2}g(\nabla_{V_{t}}\nabla_{\partial_{t}}V_{t}+\nabla_{\nabla_{\partial_{t}}V_{t}}V_{t}, \nabla_{V_{t}}V_{t})+k_{g}^{2}\di V_{t}\di \eta_{t}\Big)v_{g}\nonumber\\
&=&\int_{M}\Big(k_{g}^{2}g(\nabla_{V_{t}}\nabla_{\partial_{t}}V_{t}+\nabla_{\nabla_{\partial_{t}}V_{t}}V_{t}, \nabla_{V_{t}}V_{t})+\di (k_{g}^{2}[\di V_{t}]\eta_{t})-\eta_{t}(\grad [k_{g}^{2}\di V_{t}])\Big)v_{g}\nonumber\\
&=&\int_{M}\Big(k_{g}^{2}g(\nabla_{V_{t}}\nabla_{\partial_{t}}V_{t}+\nabla_{\nabla_{\partial_{t}}V_{t}}V_{t}, \nabla_{V_{t}}V_{t})-g(\nabla_{\partial_{t}}V_{t}, \grad (k_{g}^{2}\di V_{t}))\Big)v_{g},\nonumber
\end{eqnarray}
where $\eta_{t}(\cdot)=g(\nabla_{\partial_{t}}V_{t}, \cdot)$. We set $X=\nabla_{\partial_{t}}V_{t}|_{t=0}$. As a consequence,
\begin{equation}\label{Eq: EulerLagrangeequF I}
\frac{1}{2}\frac{d}{dt}|_{t=0}\int_{M}k_{g}^{2}\Big(g(\nabla_{V_{t}}V_{t}, \nabla_{V_{t}}V_{t})+(\di V_{t})^{2}\Big)v_{g}=\int_{M}k_{g}^{2}g(\nabla_{V}X+\nabla_{X}V, \nabla_{V}V)-g(X, \grad[k_{g}^{2}\di V])v_{g}
\end{equation}
Since $X$ is perpendicular to $V$, we have $X=f E_{1}=g(X, JV) J(V)$. By (\ref{Eq: connection M2 V}), we have
\begin{equation*}
k_{g}^{2}g(\nabla_{V}X+\nabla_{X}V, \nabla_{V}V)=k_{g}^{2}g\Big(a V(f) + a b f\Big)=\di(a k_{g}^{2} f V)+g(X, J V) V\Big(\di (J V) k_{g}^{2}\Big),
\end{equation*}
and using (\ref{Eq: connection M2 V}) and (\ref{Eq: Laplacian 2 manifold product JV}), we get
\begin{eqnarray}\label{Eq: EulerLagrangeequF II}
V\Big(\di J(V)k_{g}^{2}\Big)J(V)-g(\grad[k_{g}^{2}\di V], J V)J V&=&\Big(-a V(k_{g}^{2})- b E_{1}(k_{g}^{2}) - V(a) k_{g}^{2} - k_{g}^{2} E_{1}(b)\Big)E_{1}\nonumber\\
&=&[V, J(V)](k_{g}^{2})JV+k_{g}^{2}(\bar{\Delta}V-g(\bar{\Delta}V, V)V, J V)J V.
\end{eqnarray}
Furthermore,
\begin{eqnarray}
&\frac{d}{dt}&\int_{M}k_{g}^{2}g\Big(\nabla_{V_{t}}V_{t}-(\di V_{t})V_{t}, \grad u\Big)v_{g}=\int_{M}k_{g}^{2}g\Big(\nabla_{\partial_{t}} \nabla_{V_{t}}V_{t}-[\frac{d}{dt}\di V_{t}]V_{t}-\di V_{t}\nabla_{\partial_{t}}V_{t}, \grad u\Big)v_{g}\nonumber\\
&=&\int_{M}k_{g}^{2}g\Big(\nabla_{V_{t}}\nabla_{\partial_{t}}V_{t}+\nabla_{\nabla_{\partial_{t}}V_{t}}V_{t}, \grad u\Big)v_{g}-\int_{M}(\di \eta_{t})k_{g}^{2}V_{t}(u)v_{g}-\int_{M}k_{g}^{2}g\Big(\nabla_{\partial_{t}}V_{t}, \di V_{t}\grad u\Big)v_{g}\nonumber\\
&=&\int_{M}k_{g}^{2}g\Big(\nabla_{V_{t}}\nabla_{\partial_{t}}V_{t}+\nabla_{\nabla_{\partial_{t}}V_{t}}V_{t}, \grad u\Big)v_{g}-\int_{M}\di (k_{g}^{2}V_{t}(u)\eta_{t})v_{g}+\int_{M}\eta_{t}(\grad[k_{g}^{2}V_{t}(u)])v_{g}\nonumber\\
&-&\int_{M}k_{g}^{2}g\Big(\nabla_{\partial_{t}}V_{t}, \di V_{t}\grad u\Big)v_{g}.\nonumber
\end{eqnarray}
As a consequence,
\begin{eqnarray}\label{Eq: EulerLagrangeequF III}
\frac{d}{dt}|_{t=0}\int_{M}k_{g}^{2}g\Big(\nabla_{V_{t}}V_{t}-(\di V_{t})V_{t}, \grad u\Big)v_{g}&=&\int_{M}k_{g}^{2}g\Big(\nabla_{V}X+\nabla_{X}V, \grad u\Big)v_{g}\nonumber\\
&&+\int_{M}g\Big(X, \grad[k_{g}^{2}V(u)]-k_{g}^{2}(\di V)\grad u\Big)v_{g},
\end{eqnarray}
where $X=f E_{1}=f JV$. By using (\ref{Eq: connection M2 V}), we have
\begin{eqnarray}\label{Eq: EulerLagrangeequF IV}
k_{g}^{2}g(\nabla_{V}X+\nabla_{X}V, \grad u)&=&k_{g}^{2} J V(u) \di (f V) - a k_{g}^{2} f V(u)\nonumber\\
&=& \di\Big(k_{g}^{2} J V(u) f V\Big)-g(J V, X)V\Big(k_{g}^{2} J V(u)\Big)+k_{g}^{2} g(J V, X) V(u) \di J V,
\end{eqnarray}
and
\begin{eqnarray}\label{Eq: EulerLagrangeequF V}
&-V(k_{g}^{2}J V(u))&+k_{g}^{2}[\di(J V)]V(u)+J V(k_{g}^{2}V(u))-k_{g}^{2}(\di V) J V(u)\nonumber\\
&=&-V(k_{g}^{2})E_{1}(u)+E_{1}(k_{g}^{2})V(u)-k_{g}^{2}[V, E_{1}](u)-k_{g}^{2} a V(u)-k_{g}^{2} b E_{1}(u)\nonumber\\
&=&-V(k_{g}^{2})E_{1}(u)+E_{1}(k_{g}^{2})V(u)\nonumber\\
&=&g(J \grad u, \grad k_{g}^{2}).
\end{eqnarray}
Substituting (\ref{Eq: EulerLagrangeequF II}) in (\ref{Eq: EulerLagrangeequF I}), (\ref{Eq: EulerLagrangeequF IV}) and (\ref{Eq: EulerLagrangeequF V}) in (\ref{Eq: EulerLagrangeequF III}),  (\ref{Eq: derivative vertical bienergy conformal}) and (\ref{Eq: derivative F}) gives (\ref{Eq: conformal change biharmonic unit vector field}).
\end{proof}

\section{Biharmonic unit sections and biharmonic unit vector fields on 2-torus}
Let $d_{1}, d_{2} \in \Real^{2}$ be two linearly independent vectors and $\Gamma \subset \Real^{2}$ the lattice given by
\begin{equation*}
\Gamma=\{m d_{1} + n d_{2}: m, n \in \mathbb{Z}\}.
\end{equation*}
Consider the 2-torus $T^{2}=\Real^{2}/\Gamma$ and the natural projection $\pi: \Real^{2}\rightarrow T^{2}$. Then $\Real^{2}$ is the universal covering space of $T^{2}$. Let $J$ be the almost complex structure on $T^{2}$ induced by the fixed orientation. Let $g$ be an arbitrary Riemannian metric on $T^{2}$, $k_{g}$ the Gaussian curvature of $(T^{2}, g)$ and $\{S, W\}$ an orthonormal frame field such that $JS=W$. We give the following definitions (\cite[pages 30-31]{DragomirPerro11}):
\begin{defn}
Let $(m, n)\in \mathbb{Z}^{2}$. We set
\begin{equation*}
Per(m, n)=\{\alpha \in C^{\infty}(\Real^{2}):\alpha(\xi+d_{1})-\alpha(\xi)=2m\pi, \alpha(\xi+d_{2})-\alpha(\xi)=2n\pi, \forall \xi \in \Real^{2}\}.
\end{equation*}
An element $\alpha \in Per(m, n)$ is referred to as an \emph{$(m, n)$-semiperiodic function}.
\end{defn}
We set $\mathcal{W}=\bigcup_{(m, n)\in \mathbb{Z}^{2}}Per(m, n)$.
\begin{defn}
Let $V$ be a unit vector field on $(T^{2}, g)$. A function $\alpha:\Real^{2}\rightarrow \Real$ is called a \emph{$(S, W)$-angle function} for $V$ if
\begin{equation*}
V\circ \pi=(\cos\alpha)S\circ\pi+(\sin\alpha)W\circ\pi
\end{equation*}
everywhere on $\Real^{2}$.
\end{defn}
We prove the following Theorem:
\begin{thm}
Let $(T^{2}, g)$ be a 2-dimensional Riemannian torus. Then:\\
(i) A unit vector field $V$ on $(T^{2}, g)$ is a biharmonic unit section if and only if it is a harmonic unit section.\\
(ii) There exist biharmonic unit vector fields on $(T^{2}, g)$.
\end{thm}
\begin{proof}
(i) Let $u\in C^{\infty}(T^{2})$ and $\widetilde{g}=e^{2u}g$ be the metric on $T^{2}$ conformal to $g$ such that $(T^{2}, \widetilde{g})$ is flat. We denote by $\widetilde{\nabla}$ the Levi-Civita connection with respect to $\widetilde{g}$, $\{V_{1}, V_{2}\}$ a parallel $\widetilde{g}$-orthonormal frame field and $\widetilde{V}=e^{-u}V$. Then, we have
\begin{equation*}
\widetilde{V}=\cos\alpha V_{1}+\sin\alpha V_{2},
\end{equation*}
for a smooth function $\alpha : T^{2}\rightarrow \Real$. We compute
\begin{eqnarray}
\widetilde{\bar{\Delta}}\widetilde{V}&=&\widetilde{\Delta}(\cos\alpha)V_{1}+\widetilde{\Delta}(\sin\alpha)V_{2}=
[-V_{1}V_{1}(\cos\alpha)-V_{2}V_{2}(\cos\alpha)]V_{1}+[-V_{1}V_{1}(\sin\alpha)-
V_{2}V_{2}(\sin\alpha)]V_{2}\nonumber\\
&=&[-\sin\alpha \widetilde{\Delta}\alpha+(\cos\alpha) \|\widetilde{\grad} \alpha\|^{2}]V_{1}+[(\cos\alpha)\widetilde{\Delta}\alpha+(\sin\alpha) \|\widetilde{\grad}\alpha\|^{2}]V_{2},\nonumber\\
\widetilde{\bar{\Delta}}(J\widetilde{V})&=&-[(\cos\alpha) \widetilde{\Delta}\alpha+(\sin\alpha) \|\widetilde{\grad}\alpha\|^{2}]V_{1}+[-\sin\alpha \widetilde{\Delta}\alpha+(\cos\alpha) \|\widetilde{\grad}\alpha\|^{2}]V_{2},\nonumber
\end{eqnarray}
\begin{equation}\label{Eq: helpfulequationstildeVbiharmonicsection}
\begin{array}{lcr}
\widetilde{g}(\widetilde{\bar{\Delta}}\widetilde{V}, \widetilde{V})=\|\widetilde{\grad}\alpha\|^{2}, &\widetilde{g}(\widetilde{\bar{\Delta}}\widetilde{V}, J\widetilde{V})=\widetilde{\Delta}\alpha, & \widetilde{g}(\widetilde{\bar{\Delta}}(J\widetilde{V}), J\widetilde{V})=\|\widetilde{\grad}\alpha\|^{2},
\end{array}
\end{equation}
\begin{eqnarray}
\widetilde{\bar{\Delta}}\widetilde{V}-\widetilde{g}(\widetilde{\bar{\Delta}}\widetilde{V}, \widetilde{V})\widetilde{V}&=&-\sin\alpha \widetilde{\Delta}\alpha V_{1}+(\cos\alpha) \widetilde{\Delta}\alpha V_{2}=\widetilde{\Delta}\alpha J(\widetilde{V}),\nonumber\\
\widetilde{\bar{\Delta}}\Big(e^{2u}
(\widetilde{\bar{\Delta}}\widetilde{V}-\widetilde{g}(\widetilde{\bar{\Delta}}\widetilde{V}, \widetilde{V})\widetilde{V})\Big)&=&\widetilde{\bar{\Delta}}(e^{2u}\widetilde{\Delta}\alpha J\widetilde{V})
=\widetilde{\Delta}(e^{2u}\widetilde{\Delta}\alpha)J\widetilde{V}+e^{2u}(\widetilde{\Delta})\alpha\widetilde{\bar{\Delta}}(J\widetilde{V})-
\widetilde{\nabla}_{\widetilde{\grad}[e^{2u}\widetilde{\Delta}\alpha]}J\widetilde{V},\nonumber
\end{eqnarray}
where $J\widetilde{V}=-\sin\alpha V_{1}+\cos\alpha V_{2}$. Using Relations (\ref{Eq: helpfulequationstildeVbiharmonicsection}) and taking the inner product of (\ref{Eq: conformal change biharmonic unit section}) with $J\widetilde{V}$, we have
\begin{equation*}
\widetilde{\Delta}(e^{2u}\widetilde{\Delta}\alpha)=0.
\end{equation*}
Therefore,
\begin{equation*}
\widetilde{\Delta}\alpha=c e^{-2u},
\end{equation*}
where $c$ is a real constant. Integrating over $T^{2}$, we have
\begin{equation*}
0=\int_{T^{2}}\widetilde{\Delta}\alpha v_{\widetilde{g}}=c\int_{T^{2}}e^{-2u}v_{\widetilde{g}},
\end{equation*}
i.e. $c=0$ and $\alpha$ is a constant function on $T^{2}$. In this case, $\widetilde{V}$ is a harmonic unit section on $(T^{2}, \widetilde{g})$. By using Corollary 4.2, $V$ is a harmonic unit section on $(T^{2}, g)$.\\
(ii) Since $V_{1}, V_{2}$ are parallel vector fields on $(T^{2}, \widetilde{g})$, we have
\begin{eqnarray}
\widetilde{\nabla}_{\widetilde{V}}J\widetilde{V}&=&\Big\{-\cos^{2}\alpha V_{1}(\alpha)-\sin\alpha\cos\alpha V_{2}(\alpha)\Big\}V_{1}+\Big\{-\sin\alpha\cos\alpha V_{1}(\alpha)-\sin^{2}\alpha V_{2}(\alpha)\Big\}V_{2},\nonumber\\
\widetilde{\nabla}_{J\widetilde{V}}\widetilde{V}&=&\Big\{\sin^{2}\alpha V_{1}(\alpha)-\sin\alpha\cos\alpha V_{2}(\alpha)\Big\}V_{1}+\Big\{-\sin\alpha\cos\alpha V_{1}(\alpha)+\cos^{2}\alpha V_{2}(\alpha)\Big\}V_{2},\nonumber
\end{eqnarray}
\begin{equation}\label{Eq: helpfulequationstildeVbiharmonicvectorfield}
[\widetilde{V}, J\widetilde{V}]=\widetilde{\nabla}_{\widetilde{V}}J\widetilde{V}-\widetilde{\nabla}_{J\widetilde{V}}\widetilde{V}=-\widetilde{\grad}\alpha.
\end{equation}
Using Relations (\ref{Eq: helpfulequationstildeVbiharmonicsection}), (\ref{Eq: helpfulequationstildeVbiharmonicvectorfield}), $\widetilde{\di}(J\widetilde{\grad}u)=0$ and taking the inner product of (\ref{Eq: conformal change biharmonic unit vector field}) with $J\widetilde{V}$, we obtain that $V$ is a biharmonic unit vector field on $(T^{2}, g)$ if and only if
\begin{equation}\label{Eq: centralPDEbiharmonicvectorfield}
\widetilde{\Delta}(e^{2u}\widetilde{\Delta}\alpha)-\widetilde{\di}(k_{g}^{2}\widetilde{\grad}\alpha)=-\widetilde{\di}(k_{g}^{2}J\widetilde{\grad} u).
\end{equation}
Equation (\ref{Eq: centralPDEbiharmonicvectorfield}) can be written
\begin{equation}\label{Eq: centralPDEbiharmonicvectorfieldI}
Ph=f
\end{equation}
where $Ph=\widetilde{\Delta}(e^{2u}\widetilde{\Delta}h)-\widetilde{\di}(k_{g}^{2}\widetilde{\grad}h)$ and $f=-\widetilde{\di}(k_{g}^{2}J\widetilde{\grad} u)$ ($h \in C^{\infty}(T^{2})$). We easily notice that $P$ is an elliptic, self-adjoint operator acting on smooth functions on $T^{2}$ equipped with the metric $\tilde{g}$. Using standard ideas of the theory of elliptic operators (\cite[p.464]{Besse87}), Equation (\ref{Eq: centralPDEbiharmonicvectorfieldI}) admits solution if and only if $f\in (Ker P)^{\perp}$. On the other hand, Equation $Pv=0$, implies
\begin{equation*}
\widetilde{\Delta}(e^{2u}\widetilde{\Delta}h)=\widetilde{\di}(k_{g}^{2}\widetilde{\grad}h).
\end{equation*}
Therefore,
\begin{equation*}
\int_{T^{2}}\widetilde{\Delta}(e^{2u}\widetilde{\Delta}h) h v_{\widetilde{g}}=\int_{T^{2}}\widetilde{\di}(k_{g}^{2}\widetilde{\grad}h)h v_{\widetilde{g}}.
\end{equation*}
Hence,
\begin{equation*}
\int_{T^{2}}e^{2v}(\widetilde{\Delta}h)^{2}v_{\widetilde{g}}=-\int_{T^{2}}k_{g}^{2}\|\widetilde{\grad}h\|^{2}v_{\widetilde{g}}.
\end{equation*}
i.e. $h$ is a constant. As a consequence, $Ph=f$ has solution if and only if $\int_{T^{2}}fv_{\widetilde{g}}=0$, which holds.
\end{proof}
\begin{rem}
In \cite{Wiegm95}, G.~Wiegmink investigated the total bending of vector fields on 2-dimensional tori. More precisely, he proved that a global unit vector field $V$ is a critical point of the total bending on $(T^{2}, g)$ if and only if $\widetilde{V}=e^{-u}V$ is a critical point of the total bending on $(T^{2}, \widetilde{g})$. Theorem 5.3 (part (ii)) asserts that is no longer the case for biharmonic unit vector fields.
\end{rem}
We remind the following Lemma (\cite{DragomirPerro11}, \cite{Wiegm95}):
\begin{lem}
Let $(S, W)$ be a fixed global orthonormal 2-frame field on the 2-torus $(T^{2}, g)$. With notations as above we have:
\begin{enumerate}
\item  Given $\alpha \in C^{\infty}(\Real^{2})$, there is $\bar{\alpha}\in C^{\infty}(T^{2})$ such that $a(p)=\bar{\alpha}(\pi(p))$ for all $p \in \Real^{2}$ if and only if $\alpha \in Per(0, 0)$.
 \item  For all $(m, n)\in \mathbb{Z}^{2}$ the set $Per(m, n)$ is an affine subspace of $C^{\infty}(\Real^{2})$ and $Per(0, 0)$ is its associated vector space, in particular:
\begin{equation*}
\alpha \in Per(m, n) \Longrightarrow [\forall Y \in \X(\Real^{2}), Y(\alpha)\in Per(0, 0)].
\end{equation*}
\item For each $V \in \X^{1}(T^{2})$ there is an angle function $\alpha \in \mathcal{W}$. For fixed $V \in \X^{1}(T^{2})$, all its angle functions differ only by integer multiples of $2\pi$ and are contained in one $Per(m, n)$, $(m, n)$ depending on $(S, W)$.
\item $\alpha \in C^{\infty}(\Real^{2})$ is an angle function for a unit vector field $V\in \X^{1}(T^{2})$ if and only if $\alpha \in \mathcal{W}$.
\item Homotopy classification of unit vector fields on $T^{2}$: According to (3) to each $V\in \X^{1}(T^{2}$, one can uniquely assign a pair $htp^{(S, W)}(V)\in\mathbb{Z}^{2}$ by requiring that $htp^{(S, W)}(V)=(m, n)$ if the $(S, W)$-angle functions of $V$ are in $Per(m, n)$. Then $V_{1}, V_{2}$ are homotopic in $\X^{1}(T^{2})$, exactly if $htp^{(S, W)}(V_{1})=htp^{(S, W)}(V_{2})$. Thus the homotopy classes of $\X^{1}(T^{2})$ are classified by the elements of $\mathbb{Z}^{2}$, and we shall denote them by $\mathcal{E}_{(m, n)}^{(S, W)}$ ($(m, n)\in\mathbb{Z}^{2}$). The homotopy classes themselves do not depend on the choice of $(S, W)$, however their index $(m, n)$ does.
\end{enumerate}
\end{lem}
We end this section by computing and studying the critical points of the bienergy functional of unit vector fields on $(T^{2}, g)$ defined on the set $\mathcal{W}$ of all angle functions. Fix an orthonormal frame field $\{S, W\}$ of $(T^{2}, g)$. The Levi-Civita connection $\nabla$ of $(T^{2}, g)$ is given by
\begin{equation}\label{connection S W torus}
\begin{array}{llcrr}
\nabla_{S}S=\alpha W, & \nabla_{S}W=-\alpha S, & \nabla_{W}S=b W, &\nabla_{W}W=-b S, &\di(Z)=S(\alpha)+W(b),
\end{array}
\end{equation}
where $Z=\alpha S+b W \in \X(T^{2})$. We consider the 1-forms $\Theta_{S}, \Theta_{W}\in\Omega^{1}(T^{2})$ given by
\begin{equation*}
\begin{array}{lr}
\Theta_{S}=g(S, \cdot), & \Theta_{W}=g(W, \cdot).
\end{array}
\end{equation*}
The volume form $\omega_{T}$ is given by $\omega_{T}=\Theta_{S}\wedge \Theta_{W}\in \Omega^{2}(T^{2})$. Let $V$ be a unit vector field on $(T^{2}, g)$ and $\theta \in \mathcal{W}$ its angle function. Then,
\begin{equation*}
\widehat{V}\circ \pi=(\cos\theta) S\circ \pi+(\sin \theta) W\circ \pi,
\end{equation*}
where $\pi:\Real^{2} \rightarrow T^{2}$ is the projection. We set $\widehat{g}=\pi^{*}g$. We denote by $\widehat{S}, \widehat{W}, \widehat{V}$ vector fields on $\Real^{2}$ which are $\pi$ - related to $S, W, V$ i.e. $d\pi_{p}(\widehat{S}(p))=S(\pi(p))$ for all $ p \in \Real^{2}$. The hat symbol indicates dependence on the pullback metric $\widehat{g}$. We have
\begin{lem}
Let $V$ be a unit vector field on $(T^{2}, g)$ and $\theta$ its angle function. Then,
\begin{eqnarray}\label{Eq: Laplacian nabla hat V}
\widehat{\bar{\Delta}}\widehat{V}-\widehat{g}(\widehat{\bar{\Delta}}\widehat{V}, \widehat{V})\widehat{V}&=&\Big\{-(\sin\theta) \widehat{\Delta}\theta + (\sin\theta) [\widehat{W}(b\circ \pi)+\widehat{S}(\alpha \circ \pi)]\Big\} \widehat{S}+\nonumber\\
&&+\Big\{(\cos\theta) \widehat{\Delta}\theta - (\cos\theta) [\widehat{S}(\alpha\circ \pi)+\widehat{W}(b\circ \pi)]\Big\}\widehat{W},\\
\widehat{\nabla}_{\widehat{V}}\widehat{V}-\widehat{\di}\widehat{V}\widehat{V}&=&\Big\{-\widehat{W}(\theta)-b\circ \pi\Big\}\widehat{S}+\Big\{\alpha \circ \pi+\widehat{S}(\theta)\Big\}\widehat{W}.\nonumber
\end{eqnarray}
\end{lem}
\begin{proof}
Since $d\pi$ is an isometry, we have
\begin{equation*}
\widehat{\nabla}_{\widehat{X}}\widehat{Y}=\nabla_{X}Y \circ \pi,
\end{equation*}
where $\widehat{X}, \widehat{Y}\in \X(\Real^{2})$ are $\pi$-related to $X, Y \in \X(T^{2})$.
By using relations (\ref{connection S W torus}), we get
\begin{eqnarray}\label{Eq: homotopy biharmonic I}
\widehat{\bar{\Delta}}\widehat{S}&=&\widehat{\nabla}_{\widehat{\nabla}_{\widehat{S}}\widehat{S}}\widehat{S}-\widehat{\nabla}_{\widehat{S}}\widehat{\nabla}_{\widehat{S}}\widehat{S}+
\widehat{\nabla}_{\widehat{\nabla}_{\widehat{W}}\widehat{W}}\widehat{S}-\widehat{\nabla}_{\widehat{W}}\widehat{\nabla}_{\widehat{W}}\widehat{S}\nonumber\\
&=&(\alpha \circ \pi)\widehat{\nabla}_{\widehat{W}}\widehat{S}-(b\circ\pi)\widehat{\nabla}_{\widehat{S}}\widehat{S}-\widehat{\nabla}_{\widehat{S}}[(\alpha\circ\pi)\widehat{W}]-\widehat{\nabla}_{\widehat{W}}[(b\circ \pi)\widehat{W}]\nonumber\\
&=&\Big\{(\alpha \circ \pi)^{2}+(b\circ \pi)^{2}\Big\}\widehat{S}+\Big\{-\widehat{S}(\alpha\circ \pi)-\widehat{W}(b\circ \pi)\Big\}\widehat{W},\nonumber\\
\widehat{\bar{\Delta}}\widehat{W}&=&\Big\{\widehat{W}(b\circ \pi)+\widehat{S}(\alpha\circ \pi)\Big\}\widehat{S}+\Big\{(\alpha \circ \pi)^{2}+(b\circ \pi)^{2}\Big\}\widehat{W},\nonumber\\
\widehat{\Delta}(\cos\theta)&=&(\widehat{\nabla}_{\widehat{S}}\widehat{S})(\cos\theta)-\widehat{S}\widehat{S}(\cos\theta)+
(\widehat{\nabla}_{\widehat{W}}\widehat{W})(\cos\theta)-\widehat{W}\widehat{W}(\cos\theta),\nonumber\\
&=&(\cos\theta)\|\widehat{\grad}\theta\|^{2}-(\sin\theta) \widehat{\Delta}\theta,\nonumber\\
\widehat{\Delta}(\sin\theta)&=&(\sin\theta)\|\widehat{\grad}\theta\|^{2}+(\cos\theta)\widehat{\Delta}\theta.
\end{eqnarray}
Combining (\ref{connection S W torus}), (\ref{Eq: homotopy biharmonic I}) and $\widehat{V}=(\cos\theta) \widehat{S}+(\sin \theta) \widehat{W}$, we compute
\begin{eqnarray}
\widehat{\bar{\Delta}}\widehat{V}&=&\widehat{\Delta}(\cos\theta)\widehat{S}+(\cos\theta) \widehat{\bar{\Delta}}\widehat{S}-
2\widehat{\nabla}_{\widehat{\grad}(\cos\theta)}\widehat{S}+\widehat{\Delta}(\sin\theta)\widehat{W}+(\sin\theta)\widehat{\bar{\Delta}}\widehat{W}-
2\widehat{\nabla}_{\widehat{\grad}(\sin\theta)}\widehat{W}\nonumber\\
&=&\Big\{(\cos\theta)\|\widehat{\grad}\theta\|^{2}-(\sin\theta)\widehat{\Delta}\theta+(\cos\theta)[(\alpha\circ\pi)^2+(b\circ\pi)^{2}]+(\sin\theta)[\widehat{W}(b\circ\pi)+
\widehat{S}(\alpha\circ\pi)]\nonumber\\
&&+2\widehat{S}(\theta)(\alpha\circ \pi)\cos\theta+2\widehat{W}(\theta)(b\circ\pi)\cos\theta\Big\}\widehat{S}+\Big\{(\sin\theta)\|\widehat{\grad}\theta\|^{2}+(\cos\theta)\widehat{\Delta}\theta+
(\sin\theta)[(\alpha\circ\pi)^{2}\nonumber\\
&&+(b\circ\pi)^{2}]-(\cos\theta)[\widehat{S}(\alpha\circ\pi)+\widehat{W}(b\circ\pi)]+2\sin\theta(\alpha\circ\pi)\widehat{S}(\theta)+2\sin \theta(b\circ\pi)\widehat{W}(\theta)\Big\}\widehat{W},\nonumber\\
\widehat{g}(\widehat{\bar{\Delta}}\widehat{V}, \widehat{V})&=&\|\widehat{\grad}\theta\|^{2}+(\alpha\circ\pi)^{2}+(b\circ\pi)^{2}+2\widehat{S}(\theta)(\alpha\circ\pi)+2\widehat{W}(\theta)(b\circ\pi),\nonumber\\
\widehat{\nabla}_{\widehat{V}}\widehat{V}&=&\Big[-\sin\theta\cos\theta \widehat{S}(\theta)-\sin\theta\cos\theta (\alpha\circ\pi)-(\sin^{2}\theta)\widehat{W}(\theta)-
(b\circ\pi)\sin^{2}\theta\Big]\widehat{S}\nonumber\\
&&+\Big[\cos^{2}\theta (\alpha\circ\pi)+(\cos^{2}\theta)\widehat{S}(\theta)+\sin\theta\cos\theta (b\circ\pi)+\sin\theta\cos\theta \widehat{W}(\theta)\Big]\widehat{W},\nonumber\\
\widehat{\di} \widehat{V}&=&(\cos\theta)\widehat{\di} \widehat{S}+\widehat{S}(\cos\theta)+(\sin\theta)\widehat{\di} \widehat{W}+\widehat{W}(\sin\theta)\nonumber\\
&=&\cos\theta (b\circ\pi)-(\sin\theta)\widehat{S}(\theta)-\sin\theta (\alpha\circ\pi)+(\cos\theta)\widehat{W}(\theta).\nonumber
\end{eqnarray}
and we immediately obtain (\ref{Eq: Laplacian nabla hat V}).
\end{proof}
Let $Q=\{s d_{1}+t d_{2}\in\Real^{2}:(s, t)\in[0, 1]^{2}\}$. Then $\pi(Q)=T^{2}$ and $\pi:Q \setminus \partial Q\rightarrow T^{2}$ is injective. By Lemma 5.5, we have
\begin{eqnarray}\label{Eq: bienergy functionl angle torus}
E_{2}(V)&=&\frac{1}{2}\int_{T^{2}}\Big\{\|\bar{\Delta}V-g(\bar{\Delta}V, V)V\|_{g}^{2}+k_{g}^{2}\|\nabla_{V}V-(\di V) V\|_{g}^{2}\Big\}v_{g}\nonumber\\
&=&\int_{Q}\Big\{\|\widehat{\bar{\Delta}}\widehat{V}-\widehat{g}(\widehat{\bar{\Delta}}\widehat{V}, \widehat{V})\widehat{V}\|_{\widehat{g}}^{2}
+(k_{g}\circ\pi)^{2}\|\widehat{\nabla}_{\widehat{V}}\widehat{V}-(\widehat{\di} \widehat{V}) \widehat{V}\|_{\widehat{g}}^{2}\Big\}\pi^{*}(\omega_{T})\nonumber\\
&=&\int_{Q}\Big\{(\widehat{\Delta}\theta-[\widehat{S}(\alpha\circ\pi)+\widehat{W}(b\circ\pi)])^{2}
+(k_{g}\circ\pi)^{2}[\widehat{W}(\theta)+(b\circ\pi)]^{2}\nonumber\\
&&+(k_{g}\circ\pi)^{2}[\widehat{S}(\theta)+(\alpha\circ\pi)^{2}]\Big\}\pi^{*}(\omega_{T})\\
&=&\int_{Q}\Big\{(\widehat{\Delta}\theta-\widehat{\di}\widehat{Z})^{2}+(k_{g}\circ\pi)^{2}\|\widehat{\grad}\theta+\widehat{Z}\|_{\widehat{g}}^{2}\Big\}\pi^{*}(\omega_{T})\nonumber\\
&=&G(\theta)\nonumber,
\end{eqnarray}
where $v_{g}=2\pi^{*}(\omega_{T})$ and $\widehat{Z}=(\alpha \circ \pi)\widehat{S}+(b\circ\pi)\widehat{W}$ (\cite{Wiegm95}). In the next Theorem, we give the critical point condition for the functional $G:\mathcal{W}\rightarrow [0, +\infty)$ defined on the set $\mathcal{W}$ of all angle functions.
\begin{thm}
Let $(T^{2}, g)$ be a 2-torus, $k_{g}$ its Gaussian curvature and $\{S, W\}$ an orthonormal frame field of $(T^{2}, g)$. Then,
(i) $V \in \X^{1}(T^{2})$ is a biharmonic unit vector field if and only if every $(S, W)$-angle $\theta$ of $V$ satisfies
\begin{equation}\label{Eq: angle PDE}
\widehat{\Delta}(\widehat{\Delta}\theta)-\widehat{\di}\Big((k_{g}\circ\pi)^{2}\widehat{\grad}\theta\Big)=\widehat{\Delta}(\widehat{\di}\widehat{Z})+\widehat{\di}\Big(
(k_{g}\circ\pi)^{2}\widehat{Z}\Big),
\end{equation}
where $\widehat{Z}=(\alpha \circ \pi)\widehat{S}+(b\circ\pi)\widehat{W}$.\\
(ii)There exists a biharmonic unit vector field in each homotopy class $\mathcal{E}_{(m. n)}^{(S, W)}$.
\end{thm}
\begin{proof}
(i) Let $X:I\rightarrow \X^{1}(T^{2})$ be a smooth path in $\X^{1}(T^{2})$ such that $X_{0}=V$. We denote by $\theta_{t} \in C^{\infty}(\Real^{2}\times I)$ the angle function for $X_{t}$. We set $\dot{\theta}=\frac{\partial{\theta}}{\partial_{t}}$ and $\theta_{0}=\theta$. By using (\ref{Eq: bienergy functionl angle torus}),  we have
\begin{eqnarray}
\frac{d}{dt}G(\theta_{t})&=&2\int_{Q}\Big\{(\widehat{\Delta}\theta_{t}-\widehat{\di}\widehat{Z})\widehat{\Delta}\dot{\theta}+(k_{g}\circ\pi)^{2}\widehat{g}(\widehat{\grad}\dot{\theta}, \widehat{\grad}\theta_{t}+\widehat{Z})\Big\}\pi^{*}(\omega_{T})\nonumber\\
&=&2\int_{Q}\Big\{\widehat{\Delta}(\widehat{\Delta}\theta_{t}-\widehat{\di}\widehat{Z})\dot{\theta}+(k_{g}\circ \pi)^{2}(\widehat{\grad}\theta_{t}+\widehat{Z})(\dot{\theta})\Big\}\pi^{*}(\omega_{T})\nonumber\\
&=&2\int_{Q}\Big\{\widehat{\Delta}(\widehat{\Delta}\theta_{t}-\widehat{\di}\widehat{Z})\dot{\theta}
+\widehat{\di}[\dot{\theta}(k_{g}\circ\pi)^{2}(\widehat{\grad}\theta_{t}+\widehat{Z})]
-\dot{\theta}\widehat{\di}[(k_{g}\circ\pi)^{2}(\widehat{\grad}\theta_{t}+\widehat{Z})]\Big\}\pi^{*}(\omega_{T})\nonumber\\
&=&2\int_{Q}\Big\{\dot{\theta}\Big[\widehat{\Delta}(\widehat{\Delta}\theta_{t}-\widehat{\di}\widehat{Z})-\widehat{\di}[(k_{g}\circ\pi)^{2}(\widehat{\grad}\theta_{t}+
\widehat{Z})]\Big]\Big\}\pi^{*}(\omega_{T})\nonumber.
\end{eqnarray}
As a consequence, we get
\begin{equation}\label{Eq: variational formula torus homotopy class}
\frac{d}{dt}G(\theta_{t})|_{t=0}=2\int_{Q}\Big\{\beta \Big[\widehat{\Delta}(\widehat{\Delta}\theta-\widehat{\di}\widehat{Z})-\widehat{\di}[(k_{g}\circ\pi)^{2}(\widehat{\grad}\theta+
\widehat{Z})]\Big]\Big\}\pi^{*}(\omega_{T}),
\end{equation}
where $\beta(\xi)=\dot{\theta}(\xi, 0)$ for any $\xi\in \Real^{2}$. Then $V$ is a biharmonic unit vector field if and only if $\frac{d}{dt}G(\theta_{t})|_{t=0}=0$ for any smooth 1-parameter variation $\theta_{t}\in \mathcal{W}$ such that $\theta_{0}=\theta$. Setting
$\beta(\xi)=\widehat{\Delta}(\widehat{\Delta}\theta-\widehat{\di}\widehat{Z})(\xi)-\widehat{\di}[(k_{g}\circ\pi)^{2}(\widehat{\grad}\theta+
\widehat{Z})](\xi)$ in (\ref{Eq: variational formula torus homotopy class}), we get (\ref{Eq: angle PDE}).\\
(ii) Fix the homotopy class and consider all solutions $\theta \in Per(m, n)$  with $(m,n) \in \mathbb{Z}^{2}$. Choose any fixed $\theta_{1}\in Per(m, n)$, then $\theta-\theta_{1}\in Per(0, 0)$ and by Lemma 5.4 (Part 1) there is $\alpha \in C^{\infty}(T^{2})$ such that $\theta-\theta_{1}=\alpha\circ\pi$. Furthermore, there exists $Y\in \X(T^{2})$ such that $\widehat{\grad}\theta_{1}=Y\circ \pi$ and $\widehat{\Delta}\theta_{1}=-\di Y\circ \pi$ (\cite{Wiegm95}). Then,
\begin{equation}\label{Eq:helpful equations angle PDE}
\begin{array}{lr}
\widehat{\Delta}\theta=(\Delta\alpha-\di Y)\circ \pi, &\widehat{\di}((k_{g}\circ\pi)^{2}\widehat{\grad}\theta)=\di(k_{g}^{2}(Y+\grad\alpha))\circ\pi.
\end{array}
\end{equation}
By using (\ref{Eq:helpful equations angle PDE}), Equation (\ref{Eq: angle PDE}) is transformed into
\begin{equation}\label{Eq:helpful equations angle PDE I}
\Delta(\Delta\alpha)-\di(k_{g}^{2}\grad\alpha)=\Delta(\di(Y+Z))+\di(k_{g}^{2}(Y+Z))=f\in C^{\infty}(T^{2}).
\end{equation}
Note that $\int_{T^{2}}fv_{g}=0$. Consider the self-adjoint elliptic operator $\widehat{P}(\alpha)=\Delta(\Delta\alpha)-\di(k_{g}^{2}\grad\alpha)$. Equation (\ref{Eq:helpful equations angle PDE I}) becomes
\begin{equation}\label{Eq:helpful equations angle PDE II}
\hat{P}(\alpha)=f.
\end{equation}
By using similar arguments as in the proof of Part (ii) of Theorem 5.3, we conclude that $v\in \Ker\widehat{P}$ ($v\in C^{\infty}(T^{2})$) if and only if $v$ is a constant. Therefore, Equation (\ref{Eq:helpful equations angle PDE II}) has a solution  if and only if $\int_{T^{2}}fv_{g}=0$, which holds. Denote by $\alpha_{0}$ this solution, then $\theta=\theta_{1}+\alpha_{0}\circ\pi$ is a solution to Equation (\ref{Eq:helpful equations angle PDE I}).
\end{proof}
\begin{rem}
In general, if $(T^{2}, g)$ is not flat, the biharmonic unit vector fields we obtain in Theorem 5.3 ii) and Theorem 5.6 ii) are not harmonic. The solutions obtained in Theorem 5.3 ii) are harmonic unit vector fields (i.e. the sections found by Wiegmink) precisely when the functions $u$ and $\Delta u$ have (pointwise) collinear gradients (since $k_{g}=-\Delta u$), as for example if $u$ is an eigenfunction of the Laplacian.
\end{rem}
\begin{thm}
On a 2-torus $(T^{2}, g)$, the Hessian form of $E_{2}$ is positive semi-definite at all critical points $V$ of $E_{2}$ i.e. $\frac{d^{2}}{dt^{2}}\{E_{2}(X(t))\}_{t=0}\geq 0$ for all $C^{\infty}$-variations $X(t)$($t\in (-\epsilon, \epsilon), \epsilon > 0$) of $V$ within $\X^{1}(T^{2})$.
\end{thm}
\begin{proof}
We assume that $V$ is a biharmonic unit vector field on $(T^{2}, g)$. Let $X:I\rightarrow \X^{1}(T^{2})$ be a smooth path in $\X^{1}(T^{2})$ such that $X_{0}=V$. We denote by $\theta_{t} \in C^{\infty}(\Real^{2}\times I)$ the angle function for $X_{t}$. We set $\dot{\theta}=\frac{\partial{\theta}}{\partial_{t}}$ and $\theta_{0}=\theta$. From the proof of part (i) in Theorem 5.6, we have
\begin{equation}\label{Eq: derivative F theta t}
\frac{d}{dt}G(\theta_{t})=2\int_{Q}\Big\{\dot{\theta}\Big[\widehat{\Delta}(\widehat{\Delta}\theta_{t}-\widehat{\di}\hat{Z})-\widehat{\di}[(k_{g}\circ\pi)^{2}(\widehat{\grad}\theta_{t}+
\widehat{Z})]\Big]\Big\}\pi^{*}(\omega_{T}).
\end{equation}
Differentiating (\ref{Eq: derivative F theta t}) w.r.t the variable $t$, we get
\begin{eqnarray}\label{Eq: second derivative F theta t}
\frac{d^{2}}{dt^{2}}G(\theta_{t})&=&2\int_{Q}\Big\{\ddot{\theta}\Big[\widehat{\Delta}(\widehat{\Delta}\theta_{t}-\widehat{\di}\widehat{Z})-\widehat{\di}[(k_{g}\circ\pi)^{2}(\widehat{\grad}\theta_{t}+
\widehat{Z})]\Big]+\nonumber\\
&&+\dot{\theta}\Big[\widehat{\Delta}(\widehat{\Delta}\dot{\theta})-\widehat{\di}[(k_{g}\circ\pi)^{2}\widehat{\grad}\dot{\theta}]\Big]\Big\}\pi^{*}(\omega_{T}).
\end{eqnarray}
Evaluating (\ref{Eq: second derivative F theta t}) at $t=0$ and using (\ref{Eq: angle PDE}) (since $V$ is a biharmonic unit vector field), we have
\begin{eqnarray}
\frac{d^{2}}{dt^{2}}G(\theta_{t})|_{t=0}&=&2\int_{Q}\beta \widehat{\Delta}(\widehat{\Delta}\beta)\pi^{*}(\omega_{T})-2\int_{Q}\beta\widehat{\di}[(k_{g}\circ\pi)^{2}\widehat{\grad}\beta]\pi^{*}(\omega_{T})\nonumber\\
&=&2\int_{Q}(\widehat{\Delta}\beta)^{2}\pi^{*}(\omega_{T})-2\int_{Q}\widehat{\di}\Big(\beta(k_{g}\circ\pi)^{2}\widehat{\grad}\beta\Big)\pi^{*}(\omega_{T})\nonumber\\
&&+2\int_{Q}(k_{g}\circ\pi)^{2}\|\widehat{\grad}\beta\|_{\widehat{g}}^{2}\pi^{*}(\omega_{T})\nonumber\\
&=&2\int_{Q}(\widehat{\Delta}\beta)^{2}\pi^{*}(\omega_{T})+2\int_{Q}(k_{g}\circ\pi)^{2}\|\widehat{\grad}\beta\|_{\hat{g}}^{2}\pi^{*}(\omega_{T})\geq 0,\nonumber
\end{eqnarray}
where $\beta(\xi)=\dot{\theta}(\xi, 0)$ for any $\xi\in \Real^{2}$.
\end{proof}
\begin{rem}
In \cite{Wiegm95}, Wiegmink proved that harmonic unit vector fields on 2-dimensional torus $(T^{2}, g)$ are stable. Theorem 5.7 shows that this property also holds for biharmonic unit vector fields on $(T^{2}, g)$.
\end{rem}

\end{document}